\documentclass[11pt]{amsart}
\usepackage{amssymb,mathrsfs,graphicx,enumerate,color}
\usepackage{comment}
\usepackage{colortbl}
\definecolor{black}{rgb}{0.0, 0.0, 0.0}
\definecolor{red}{rgb}{1.0, 0.5, 0.5}

\title[   ]{Asymptotic behavior toward viscous shock for the outflow problem of barotropic Navier-Stokes equations}

\author[Kang]{Moon-Jin Kang}
\author[Oh]{HyeonSeop Oh$^{*}$}
\author[Wang]{Yi Wang}
\address[Moon-Jin Kang]{%\newline Laboratoire Jacques-Louis Lions, 
\newline Department of Mathematical Sciences, \newline Korea Advanced Institute of Mathematical Sciences, Daejeon 34141, Korea}
\email{moonjinkang@kaist.ac.kr}
\address[HyeonSeop Oh]{
\newline Department of Mathematical Sciences, \newline Korea Advanced Institute of Mathematical Sciences, Daejeon 34141, Korea}
\email{ohs2509@kaist.ac.kr}
\address[Yi Wang]{
\newline State Key Laboratory of Mathematical Sciences and Institute of Applied Mathematics, AMSS, CAS, Beijing 100190, P. R. China \newline and School of Mathematical Sciences, University of Chinese Academy of Sciences, Beijing 100049, P. R. China
}
\email{wangyi@amss.ac.cn}

\newtheorem{theorem}{Theorem}[section]

\newtheorem{remark}{Remark}[section]

\newcommand{\e}{\varepsilon}

\numberwithin{figure}{section}
\newcommand{\beq}{\begin{equation}}
\newcommand{\eeq}{\end{equation}}
\newcommand{\bsp}{\begin{split}}
\newcommand{\esp}{\end{split}}

\newcommand{\RR}{{\mathbb R}}

\newcommand{\s}{\sigma}

\newcommand\adots{\mathinner{\mkern2mu\raise1pt\hbox{.}
\mkern3mu\raise4pt\hbox{.}\mkern1mu\raise7pt\hbox{.}}}

\newtheorem{theo}{Theorem}[section]
\newtheorem{prop}[theo]{Proposition}

\newtheorem{lem}[theo]{Lemma}

\newtheorem{rem}[theo]{Remark}

\def\charf {\mbox{{\text 1}\kern-.30em {\text l}}}
\newcommand{\util}{{\tilde{u}}}
\def \l {\lambda}

\def \nb {\nabla}
\def \m {\mu}
\def \rd {\partial}

\def \d {\delta}

\def \O {\Omega}

\def \r {\rho}
\def \rhotil {\tilde{\rho}}
\def \b {\beta}
\def \g {\gamma}
\def \Rp {\mathbb{R}_+}
\def \intRp {\int_{\mathbb{R}_+}}
\newcommand{\norm}[1]{\left\|#1\right\|}
\def \Gnew  {G^{\text{new}}}

\def \UtilX {\widetilde{U}^{X, \beta}}
\def \utilX {\tilde{u}^{X, \beta}}
\def \rhotilX {\tilde{\rho}^{X, \beta}}

\def \uubar {\underline{u}}
\def \rubar {\underline{\rho}}
\def \Util {\widetilde{U}}
\def \ptil {\tilde{p}}
\begin{document}
%%%%%%%%%%%%%%%%
\bibliographystyle{acm}
%plain vs acm?

\date{\today}


\thanks{\textbf{Acknowledgment.} M.-J. Kang and H. Oh were partially supported by the National Research Foundation of Korea  (RS-2024-00361663  and NRF-2019R1A5A1028324). Yi Wang was supported by NSFC (Grant No. 12421001 and 12288201) and CAS Project for Young Scientists in Basic Research (Grant No. YSBR-031).
}

\begin{abstract}
We study the time-asymptotic stability of viscous shock profile for the outflow problem of the barotropic Navier-Stokes equations on the half-line. 
We consider the case where the far-field state, as the right end state of the 2-Hugoniot curve, belongs to the subsonic region or the transonic curve.
We employ the method of $a$-contraction with shifts to prove that, if both the shock strength and the initial perturbation are suitably small, and the viscous shock is far from the outflow boundary, then the solution asymptotically converges to the viscous shock up to a dynamical shift. We also prove that the speed of the time-dependent shift decays to zero as time goes to infinity, so that the shifted viscous shock still retains its original profile asymptotically.

Since the outflow problem in the Lagrangian mass coordinate leads to a free-boundary problem due to the absence of a boundary condition for the fluid density, we consider the outflow problem in the original Eulerian coordinate instead. Although the method of $a$-contraction with shifts is technically more complicated in the Eulerian coordinate than in the Lagrangian one, this provides a more favorable framework by avoiding the difficulties arising from a free boundary.
Note that this is the first result on the time-asymptotic stability of viscous shock for the outflow problem of the Navier-Stokes equations.

\end{abstract}

\maketitle \centerline{\date}
{\renewcommand{\thefootnote}{}\footnotetext{* Corresponding author.}}

\tableofcontents

\section{Introduction}
\setcounter{equation}{0}
We consider the outflow problem for the one-dimensional compressible barotropic Navier-Stokes equations on the half-line $\mathbb{R}_+ := (0, +\infty)$:
\begin{equation}
\begin{aligned}\label{NS}
&\r_t + (\r u)_x = 0, \, \, & t \in \mathbb{R}_+, \, \, x \in \mathbb{R}_+,\\
&(\r u)_t + (\r u^2 + p(\rho))_x = \m u_{xx}. &
\end{aligned}
\end{equation}
Here, $\rho(t,x)$ and $u(t,x)$ denote the density and velocity of the fluid, respectively, $p(\rho)$ is the pressure, and $\mu>0$ is the  constant viscosity coefficient. We consider the classical $\gamma$-law pressure
\[
p(\rho) = K \r^{\g}, \quad \g > 1,
\]
where $\g>1$ is the adiabatic constant and $K>0$ is a constant depending on the fluid. For simplicity, we normalize $K = \mu = 1$.

We impose the initial condition \begin{equation*}%\label{IC}
(\rho, u)(0,x) = (\rho_0(x), u_0(x)), \quad x \in \mathbb{R}_+,
\end{equation*}
which satisfies the non-vacuum initial density and the constant far-field condition:
\begin{equation}\label{FFC}
\inf_{x\in \mathbb{R}} \r_0(x)  >0,\quad (\r_0(x), u_0(x)) \to (\r_+, u_+), \quad \text{as }\ x\to +\infty ,\quad \r_+>0,
\end{equation}
and the outflow boundary condition
\begin{equation}\label{BC}
 u (t,0)=  u_- < 0 , \quad t>0.
\end{equation}

Note that under the outflow boundary condition \eqref{BC}, the density on the boundary $\rho(t,0)$ cannot be prescribed due to  the well-posedness of the initial-boundary value problem. This is in contrast to the inflow boundary condition, corresponding to the case $u_->0$, where the density on the boundary $\rho(t,0)$ should be imposed. We remark that for the impermeable wall problem $u_-=0$ in \eqref{BC}, the density on the boundary also cannot be prescribed.\\

There have been extensive studies on the initial-boundary value problems (IBVPs) for both inviscid and viscous conservation laws. For the inviscid case, Bardos-Le Roux-N\'ed\'elec \cite{BRN79} established the existence of entropy solutions (in the sense of Kru{\v{z}}kov) to the IBVP for scalar conservation laws using a vanishing viscosity method. Recently, based on the boundary Riemann problem and boundary layer analysis in \cite{BS20}, Ancona-Marson-Spinolo \cite{AMS24} established the existence of small BV solutions to the system of conservation laws, by considering the boundary condition coming from the inviscid limit of hyperbolic-parabolic systems. We also mention the work of Hashimoto-Matsumura \cite{HM24}, which studied the inviscid limit of radially symmetric stationary solutions for the inflow and outflow problems of compressible Navier-Stokes equations. For further studies on IBVPs in conservation laws, we refer to \cite{AMA97,AC97,BS09,DL88,KSX97}.

While the above results focus on the inviscid system and the inviscid limit, the present paper concerns the viscous system, specifically the time-asymptotic stability of solutions to the outflow problem \eqref{NS}-\eqref{BC}. It is well-known that the time-asymptotic behavior of the solution to  IBVPs \eqref{NS}-\eqref{BC} is closely related to the Riemann problem of the corresponding inviscid Euler equations (\eqref{NS} with $\mu=0$) with the right end state $(\r_+, u_+)$ in \eqref{FFC} and all the possible left end states $(\r_-, u_-)$ consistent with the boundary condition \eqref{BC} at the boundary $x=0$. 
More specifically, as discussed by Matsumura \cite{MatBVP}, for the given states $(\r_+, u_+)$ and $u_-$ in the outflow problem \eqref{NS}-\eqref{BC}, the asymptotic profile is determined as follows. If there exists $\r_->0$ such that 
$(\r_-, u_-)$ and $(\r_+, u_+)$ are connected by a Riemann solution as a right-going simple wave (i.e., moving from left to right), then the corresponding wave is the appropriate‌  asymptotic profile. If such a $\r_->0$ does not exist, then the boundary layer solution (or called stationary solution) may emerge to compensate the gap.
In this situation, the time-asymptotic profile is given by either the boundary layer solution or a superposition of the boundary layer solution and the right-going waves. Based on this criterion,
Matsumura \cite{MatBVP} and Kawashima-Zhu \cite{KZ-09} classified all possible asymptotic patterns for the outflow problem \eqref{NS}-\eqref{BC} according to the sign of $u_+$ and the transonic, subsonic, or supersonic region of the right end state $(\r_+, u_+)$. In summary, the asymptotic patterns of the outflow problem \eqref{NS}-\eqref{BC} can be described by one of the following waves: a viscous shock wave, a boundary layer solution, a rarefaction wave, or a superposition of a boundary layer solution and a rarefaction wave. 

For the rigorous proof of the time-asymptotic stability of these wave patterns for the outflow problem \eqref{NS}-\eqref{BC}, Kawashima-Nishibata-Zhu \cite{KNZ} first proved the case of transonic and supersonic boundary layer solutions. Moreover, Costanzino-Humpherys-Nguyen-Zumbrun \cite{CHNZ09} established the spectral stability of boundary layer solutions; this result further implies the nonlinear stability, as shown in Nguyen-Zumbrun \cite{NZ09}. For further studies on the boundary layer solutions in one-dimensional or multi-dimensional settings, we refer to \cite{GMGK10, Toan10, NZ10,WYY2025}. Then Kawashima-Zhu \cite{KZ-09} proved the stability of a single rarefaction wave and furthermore, they \cite{KZ-08} proved the case of the superposition of a boundary layer solution and a rarefaction wave by delicate analysis of the boundary terms. 
However, to the best of our knowledge, there are no existing results on the time-asymptotic stability of viscous shock waves for the outflow problem \eqref{NS}-\eqref{BC}, due to the essential difficulties arising from the outflow boundary. 

The anti-derivative technique, as a classical method for the stability of viscous shock waves, is powerful in proving the stability of a single viscous shock wave for the Cauchy problem. However, there is no exact information for the anti-derivative variables on the boundary in the half-space problem for inflow, outflow and impermeable wall cases. Therefore, it is difficult to prove the time-asymptotic stability of viscous shock waves for the half-space problem by using the anti-derivative technique directly. In both impermeable wall and inflow cases, the half space problem in Eulerian coordinate can be transformed into a new half space problem in Lagrangian coordinate. 
Concerning the half-space problem in Lagrangian mass coordinate, the constant shift arguments invented by Matsumura-Mei \cite{MM} for the stability of viscous shocks can be utilized to obtain the exponential decay in time for the velocity perturbation anti-derivative variable on the impermeable wall boundary, which is enough to handle the boundary term for the energy stability analysis in the  anti-derivative regime. Then, this idea is extended to the inflow problem to prove the stability of a single viscous shock and even its superposition with the boundary layer solution by Huang-Matsumura-Shi \cite{HMS} based on the classical anti-derivative techniques and the constant shift arguments in \cite{MM}.

However, to the best of our knowledge, these techniques can not be applicable to the outflow problem \eqref{NS}-\eqref{BC}. This is because the approach used to define the anti-derivative variables and the shift in \cite{MM} relies crucially on the mass equation \eqref{NS}$_1$ and the boundary values of both fluid density and velocity, while the density information at the boundary is absent for the outflow problem \eqref{NS}-\eqref{BC}.

%However, the outflow problem \eqref{NS}-\eqref{BC} in Eulerian coordinates becomes a free-boundary problem in Lagrangian mass coordinates due to the absence of a density boundary condition, which makes the analysis of \eqref{NS}-\eqref{BC} more complicated in the Lagrangian frame. Thus, it is more natural to investigate the outflow problem \eqref{NS}-\eqref{BC} in Eulerian coordinates, and the anti-derivative techniques with the shift argument in \cite{MM} are somewhat limited in proving the viscous shock stability for the outflow problem.

In this paper, we aim to study the long-time behavior of the solution $(\r, u)$ to the outflow problem \eqref{NS}-\eqref{BC} toward the right-going viscous shock (i.e., the shock speed is positive), when the viscous shock profile has a small jump amplitude and is sufficiently far from the boundary.
To this end, we employ the method of $a$-contraction with shifts invented by Kang-Vasseur \cite{Kang-V-1, KV21} for the $L^2$ contraction property of viscous shock and Kang-Vasseur-Wang \cite{KVW23,KVW-NSF} for the time-asymptotic stability of viscous shock and the composite waves. Based on $a$-contraction method, the $L^2$ perturbation of the viscous shock is controlled (see \cite{EEK25,HKK23}). Furthermore, the boundary terms in the energy estimates are controlled by the exponential decay of the viscous shock in Lemma~\ref{lem:VS}, and by the outflow boundary condition $u_-<0$ (see Section~\ref{sec:bd} and \eqref{est:rhox}).

For the right-going weak shock as an asymptotic profile, we consider the case $u_+ < u_-$ and the right end state $(\r_+, u_+)$ lies in the subsonic or transonic region, i.e.,
\begin{equation}\label{cond:U+}
(\r_+, u_+) \in \O_{sub}^- \cup  \Gamma_{trans}^-,
\end{equation}
where $\O_{sub}^- := \{(\r, u)|  \lambda_2(\rho,u):=u + \sqrt{p'(\rho)} > 0, \, \, u < 0\} $, and $\Gamma_{trans}^- := \{(\r, u)|  \lambda_2(\rho,u):=u + \sqrt{p'(\rho)} = 0\}$.
For a given $(\r_+, u_+) \in \O_{sub}^- \cup  \Gamma_{trans}^-$ with $u_+ < u_- < 0$, the values of $\r_-$ and the shock speed $\s$ can be uniquely determined by the following Rankine-Hugoniot (RH) conditions:
\begin{equation}
\begin{aligned}\label{RH}
&-\s (\r_+ - \r_-) + (\r_+ u_+ - \r_- u_-) = 0,\\
&-\s (\r_+ u_+ - \r_- u_-) + (\r_+ u_+^2 - \r_- u_-^2 + p(\r_+) - p(\r_-)) = 0,
\end{aligned} 
\end{equation}
% RH conditions \eqref{RH}. 
 and the Lax entropy condition:
\beq\label{lax}
\lambda_2(\rho_+,u_+) < \s < \lambda_2(\rho_-,u_-).
\eeq
%%then there exists unique $\r_->0$ such that $(\r_-, u_-)$ is located in the 2-shock curve,
% i.e.,
%\[
%    (\r_-, u_-) \in S_2(\r_+, \u_+)
%\]
%If we further assume that $0>u_- > u_+$, the solution is expected to converge toward the viscous 2-shock wave.  
%Thus, the second eigenvalue is non-negative on $\O_{sub}^- \cap \Gamma_{trans}^-$ as
%\beq\label{la2}
%,\quad \mbox{ on } \O_{sub}^- \cap \Gamma_{trans}^-.
%\eek
From those conditions, the shock speed $\s$ of 2-shock wave is given by 
\[
\s = u_+ + \sqrt{\frac{\r_-}{\r_+}}\sqrt{\frac{p(\rho_+) - p(\rho_-)}{\rho_+ - \rho_-}}>0.
\]
%In addition, by \eqref{la2}we have $\s > 0$.

 %connecting $(\r_\pm, u_\pm)$ uniquely exists up to the constant translation% if the end states $(\r_\pm, u_\pm)$ satisfy the 
%From the boundary condition \eqref{BC}, the shock speed $\s$ must be a positive value. Moreover, since the end states satisfy \eqref{cond:U+}, the corresponding wave should be viscous 2-shock.

The viscous 2-shock profile $(\rhotil(\xi), \util(\xi))$ connecting two states $(\r_\pm, u_\pm)$ with $\xi = x - \s t$ is given by the following ODE system:
\begin{equation}
\begin{aligned}\label{Vshock}
&-\s \rhotil' + (\rhotil \util)' = 0, & \text{where} \quad  ':= \frac{d}{d\xi},\\
&-\s (\rhotil \util)' +(\rhotil \util^2 + p(\rhotil))' = \util'',&  (\rhotil, \util)(\pm\infty)=(\r_\pm, u_\pm). 
\end{aligned} 
\end{equation}
We aim to prove the time-asymptotic stability of this viscous shock $(\rhotil(\xi), \util(\xi))$ for the outflow problem \eqref{NS}-\eqref{BC} on the half-line when the right end state $(\r_+,u_+)$ lies in the subsonic or transonic region.\\

It should be emphasized that a single viscous 2-shock wave $(\rhotil(\xi), \util(\xi))$ can also be an asymptotic profile when the right end state $(\r_+,u_+)$ lies in the supersonic region, that is, $(\r_+,u_+)\in \O_{super}^-:= \{(\r, u)|  \lambda_2(\rho,u):=u + \sqrt{p'(\rho)} < 0\}$, in addition to the case in \eqref{cond:U+} considered in the present paper.
This supersonic case can happen when the 2-Hugoniot curve starting from $(\r_+, u_+)\in \O_{super}^-$ crosses the transonic curve and intersects theboundary layer line starting from the right end state $(\r_+,u_+)$. In addition, the  shock speed connecting $(\r_-, u_-)\in\O_{sub}^-$ to $(\r_+, u_+)\in \O_{super}^-$ should be nonnegative and the shock wave strength is not necessarily weak. However, since our approach is based on the method of $a$-contraction with shifts, which is more efficient for the stability of weak shocks, in particular, to compressible Navier-Stokes equations \eqref{NS}, the stability of this large viscous 2-shock wave for $(\r_+,u_+)\in \O_{super}^-$ and $(\r_-, u_-)\in\O_{sub}^-$  is still open.
%whose strength depends on the given right end state, this case cannot be handled in general. In particular, for such weak shocks, the 2-Hugoniot curve starting from the right state in $\O_{super}^-$ cannot connect to the left state in $\O_{sub}^-$. 
We leave the study of this case to the future work on the stability of large shocks.  

%However, whenever we consider a weak shock, since the small jump strength depends on the given right state, the 2-Hugoniot curve starting from the right state in $\O_{super}^-$ cannot reach the region $\O_{sub}^-$ that a left state is in.
%Thus, this case cannot be handled by our approach based on the method of $a$-contraction with shifts that can be used for the stability of a weak viscous shock. 
%We leave this case for the future study of the stability of large shock.  

\subsection{Main result} We now state the main result on the global existence and long-time behavior toward the viscous 2-shock wave of the outflow problem \eqref{NS}-\eqref{cond:U+}.

\begin{theorem}\label{thm:main}
Assume $\g > 1$. For a given constant $(\r_+, u_+)$ satisfying $u_+ <0$ and \eqref{cond:U+}, there exist positive constants $\d_0, \e_0$ such that the following holds.

For any $u_- < 0 $ satisfying $u_->u_+$ and $|u_- - u_+| < \d_0$, let $\r_->0$ be the (unique) constant state such that $(\r_-, u_-)$ lies on the 2-Hugoniot curve. 
Let $(\rhotil(x - \s t), \util(x - \s t) )$ be the viscous 2-shock wave in \eqref{Vshock} connecting $(\r_-, u_-)$ to $(\r_+, u_+)$, with $\rhotil(0)= (\rho_- + \rho_+)/2$. 
Then, there exists a sufficiently large $\beta>0$ (depending only on the shock strength $|u_- - u_+|$) such that the following holds.
Let $(\r_0, u_0)$ be any initial data such that 
\begin{equation}\label{initial perturbation}
\norm{(\r_0, u_0) - (\r_+, u_+)}_{L^2(\b, \infty)} + \norm{(\r_0, u_0) - (\r_-, u_-)}_{L^2(0, \b)} + \norm{(\rd_x \r_0, \rd_x u_0)}_{L^2(\Rp)} < \e_0.
%\norm{(\r_0, u_0)(\cdot) - (\rhotil, \util)(\cdot-\b)}_{H^1(\Rp)}+\b^{-1}\leq \e_0.
\end{equation}
Then, the outflow problem \eqref{NS} - \eqref{BC} %subject to the initial data $(\r_0, u_0)$ 
admits a unique global solution $(\r, u)(t,x)$. Moreover, there exists a Lipschitz shift $t \mapsto X(t)$ such that
\[
\begin{aligned}
    &(\r, u)(t,x) - (\rhotil, \util)(x - \s t - X(t) - \b) \in C([0,\infty);H^1(\Rp)),\\
    &u_{xx}(t,x) - \util_{xx}(x - \s t - X(t) - \b) \in L^2(0,\infty;L^2(\Rp)),
\end{aligned}
\]
and the solution asymptotically converges to the viscous shock wave
%In addition, as , we have
\begin{equation}\label{asym-U}
    \sup_{x \in \Rp} \big|(\r, u)(t,x) - (\rhotil, \util)(x - \s t - X(t) - \b) \big| \to 0, \quad {\it as}\quad t \to +\infty,
\end{equation}
with the speed of the shift tending to zero:
\begin{equation}\label{asym-X}
    \lim_{t \to \infty} |\dot{X}(t)| = 0.
\end{equation}
\end{theorem}
\begin{remark}
Theorem \ref{thm:main} requires the strength of the viscous shock wave to be small, bounded by a constant $\delta_0$ that depends on the right end state $(\rho_+, u_+)$. Here, $\delta_0$ is independent of $\e_0$, which represents the smallness of the initial perturbation.
\end{remark}
\begin{remark}
We fix the initial shock position %, as $\rhotil(\beta)= (\rho_- + \rho_+)/2$,
far from the boundary by choosing sufficiently large $\beta>0$. This choice is technically necessary to control the size of the perturbation at the boundary as small as we want (see Lemma~\ref{lem:bd}).

Here, the constant $\beta$ depends only on the shock strength (and is independent of $\e_0$). Indeed, the boundary effect is of order
$e^{-C\delta \beta}$ (see \eqref{est:pri}), and thus choosing $\beta$ sufficiently large makes this
quantity sufficiently small (see \eqref{def:y0} and \eqref{smalle}).
\end{remark}
\begin{remark}
Theorem \ref{thm:main} establishes a first result on the time-asymptotic stability toward viscous shocks up to a dynamical shift $X(t)$ to the outflow problem \eqref{NS}-\eqref{BC}. 
%Moreover, thanks to \eqref{asym-X}, the speed of shift does not affect the shock speed in the limit as $t \to \infty$. Thus, the shifted wave $(\rho^S,u^S)(x-\sigma t-X(t)-\beta)$ time-asymptotically converges to the shock profile $(\rho^S,u^S)(x-\sigma t)$ up to constant shift, although we do not know the exact location of the viscous shock. 
Moreover, thanks to \eqref{asym-X}, the speed of the shifted shock converges to the original shock speed. Thus, it can be expected that the shifted shock wave asymptotically converges to the family of viscous shock profiles, although we do not know the exact location of the viscous shock. In fact, even for Cauchy problem, the time-asymptotic shock location of $\displaystyle \lim_{t\rightarrow+\infty} X(t)$ may not exist under the initial perturbation \eqref{initial perturbation} with the infinite mass. Therefore, we can also expect that the time-dependent shift $X(t)$ is essential for the shock stability to the outflow problem with the infinite mass in the half space.
\end{remark}
\begin{remark}
The proof of Theorem \ref{thm:main} is based on the method of $a$-contraction with shifts. This allows to control the $L^2$ perturbation to be contractive in time, up to small error by the boundary effect. We apply the method of $a$-contraction with shifts to the outflow problem in Eulerian coordinates, instead of handling a free-boundary problem in Lagrangian mass coordinates. 
Notice that the use of method of $a$-contraction with shifts in Eulerian coordinates is technically more complicated than in Lagrangian mass coordinates. That is why the existing results based on the method of $a$-contraction with shifts were handled in Lagrangian mass coordinates as in \cite{G-CKV,HKK23,HKKKO25,HKK25,HWWW,Kang23Inflow,HL25,KV-Inven,KV-2shock,KVW23,KVW-NSF}. 
\end{remark}

\section{Preliminaries}
\setcounter{equation}{0}
In this section, we present useful estimates for viscous shock waves. Moreover, we also provide a Poincar\'e-type inequality with the optimal constant on any bounded interval.
\subsection{Viscous shock waves}
Here, we provide several properties of viscous 2-shock waves $(\rhotil, \util)(\xi)=(\rhotil, \util)(x-\s t)$ with small amplitude. Without loss of generality, we assume that $\rhotil(0)=\frac{\r_- + \r_+}{2}$. The proof of the lemma below can be found in \cite{KV21}.
\begin{lem}\label{lem:VS}
For any state $(\r_+, u_+)$, there exists a constant $C>0$ such that the following is true. For any left end state $(\r_-, u_-)$ lying on the 2-Hugoniot curve, there exists a unique solution $(\rhotil, \util)(\xi) = (\rhotil, \util)(x-\s t)$ to \eqref{Vshock} such that $\rhotil(0)=\frac{\r_- + \r_+}{2}$. Let $\d>0$ denote the shock strength as $\d := |u_- - u_+| \sim |\r_- - \r_+|$. Then, it holds that 
\[
\rhotil_\xi < 0, \quad \util_\xi <0,
\]
and 
\begin{align*}
    &|\util - u_\pm| \leq C\d e^{-C\d |\xi|}, \quad \pm \xi < 0,\\
    %&|\util - u_-| \leq C\d e^{-C\d |\xi|}, \quad \xi <0,\\
    %&|\util - u_+| \leq C\d e^{-C\d |\xi|}, \quad \xi >0,\\
    &|(\rhotil_\xi, \util_\xi)| \leq C\d^2 e^{-C\d |\xi|}, \quad \forall \xi \in \RR,\\
    &|(\rhotil_{\xi \xi}, \util_{\xi \xi})|  \leq C\d |(\rhotil_\xi, \util_\xi)|, \quad \forall \xi \in \RR.
\end{align*}
\end{lem}
\begin{lem}
For any state $(\r_+, u_+)$, there exist positive constants $C$ and $\delta_0$ such that the following holds. Let $(\rhotil, \util)(\xi)$ be a viscous 2-shock wave connecting $(\r_-, u_-)$ to $(\r_+, u_+)$. If $\d:= |u_- - u_+| < \d_0$, then 
\beq\label{pos2}
    |\s -\util - \sqrt{p'(\rhotil)} | \le C \delta.
    \eeq  
    In particular, this implies $\s-\util>0$.
\end{lem}
\begin{proof}
    \eqref{pos2} follows directly from Lemma~\ref{lem:VS} and the Lax entropy condition \eqref{lax}, provided that $\delta_0$ is sufficiently small. Moreover, for sufficiently small $\delta_0$, we obtain
    \[
    \s-\util \geq \frac{\sqrt{p'(\r_+)}}{2} >0.
    \]
\end{proof}
\begin{comment}
    \begin{remark}\label{rem:pos}
    By Lemma~\ref{lem:VS} and the Lax entropy condition \eqref{lax}, the velocity component of the viscous 2-shock with a small amplitude satisfies
    \beq\label{pos2}
    |\s -\util - \sqrt{p'(\rhotil)} | \le C \delta,
    \eeq  
    and so $\s-\util>0$.
    \end{remark}
\end{comment}

\subsection{Poincar\'e-type inequality}
In this subsection, we introduce the Poincar\'e-type inequality, which is useful to get the $L^2$-type energy estimate. The proof of the lemma below can be found in \cite{Kang23Inflow}.
\begin{lem}\label{poincare} \cite{Kang23Inflow}
For any $a<b$ and function $f:[a,b] \to \RR$ satisfying $\int_{a}^b (y-a)(b-y)|f'(y)|^2\,dy < \infty$,
\[
\int_a^b \left|f(y) - \frac{1}{b-a} \int_a^b f(y)\,dy\right|^2\,dy \leq \frac{1}{2}\int_a^b (y-a)(b-y)|f'(y)|^2\,dy.
\]
\end{lem}

\section{A priori estimates and proof of Theorem \ref{thm:main}}
\setcounter{equation}{0}
In this section, we present a priori estimates for the $H^1$ perturbation of the solution, and outline the proof of the global existence and long-time behavior of the outflow problem toward viscous shock waves.
\subsection{Local existence of solutions} First, we provide the local existence of strong solutions to the outflow problem. Since the local existence of the solution can be obtained by the standard iteration argument, we omit the proof (see, for example, \cite{C20,Kang23Inflow}).
\begin{prop}\label{prop:loc} (Local existence) 
    For any constant $\b > 0 $, let $\rubar, \uubar$ be smooth monotone functions such that
    \[
    (\rubar(x), \uubar(x)) = (\r_+, u_+), \quad \text{for } x \geq \b, \quad \rubar(0) > 0.
    \]
    For any constants $M_0, M_1, \underline{\kappa}_0, \bar{\kappa}_0,  \underline{\kappa}_1, \bar{\kappa}_1$ with $0<M_0 < M_1$ and $0 < \underline{\kappa}_1 < \underline{\kappa}_0 < \bar{\kappa}_0 < \bar{\kappa}_1$, there exists a constant $T_0>0$ such that if 
    \begin{align*}
        &\norm{(\r_0 - \rubar, u_0 - \uubar)}_{H^1(\Rp)} \leq M_0,\\
        &0 < \underline{\kappa}_0 \leq \r_0(x) \leq \bar{\kappa}_0, \quad x \in \Rp,
    \end{align*}
    the outflow problem \eqref{NS} - \eqref{BC} has a unique solution $(\r, u)$ on $[0,T_0]$ such that 
\[
\r - \rubar \in C([0,T_0];H^1(\Rp)), \quad u-\uubar \in C([0,T_0];H^1(\Rp)) \cap L^2(0,T_0;H^2(\Rp)),
\]
and
\[
    \norm{(\r - \rubar, u - \uubar)}_{H^1(\Rp)} \leq M_1.
\]
Moreover, we have 
\[
    \underline{\kappa}_1 \leq \r(t,x) \leq \bar{\kappa}_1, \quad u(t,0) = u_-, \quad ^\forall t>0, \quad x \in \Rp.
\]
\end{prop}

\subsection{Construction of weight and shift}
In this subsection, we construct a weight function and a shift function. For simplicity, we use the following notation and omit the dependence on $X,\b$:
\begin{align*}
    a(t,x) := a^{X,\b}(t,x) = a(x  - \s t - X(t) - \b),\\
    \rhotil(t,x):= \rhotilX(t,x) = \rhotil(x-\s t - X(t) - \b),\\
    \util(t,x):= \utilX(t,x) = \util(x-\s t - X(t) - \b),
\end{align*}
where the shift function $X(t)$ and the positive constant $\beta$ will be defined later.

First, we define the weight function $a(t,x) = a(\xi)$ as 
\begin{equation}\label{def:a}
    a(\xi) := 1 +\sqrt{\d} + \frac{u_+ - \util(\xi)}{\sqrt{\d}}, \quad \xi = x - \s t.
\end{equation}
Then, the weight function satisfies $1 \leq a \leq 1+ \sqrt{\d} \leq \frac{3}{2}$, and 
\begin{equation}\label{der-a}
    a'(\xi) = -\frac{\util'(\xi)}{\sqrt{\d}} >0.
\end{equation}
Now, we define a shift function $X(t)$ as a solution of ODE:
\begin{equation}\label{def:X}
    \begin{aligned}
        &\dot{X}(t) = -\frac{M}{\d}\left[\intRp a^{X,\b} \frac{p(\rhotilX)}{(\s -\utilX)}\rhotilX_x(u-\utilX)\,dx + \intRp a^{X,\b} \rhotilX (u - \utilX)\utilX_x\,dx\right],\\
        &X(0)=0,
    \end{aligned}
\end{equation}
where $u$ is a solution to the outflow problem \eqref{NS} - \eqref{BC}, $M = \frac{2(\g+1)}{\r_+}$, and $\b>0$ is the positive constant which will be chosen later.

The existence and the Lipschitzness of the shift $X(t)$ can be found in \cite{KVW23}.
\subsection{A priori estimates}
\begin{prop}\label{p-est}
    For given $(\r_+, u_+) \in \mathbb{R}_+ \times \mathbb{R}$ satisfiying $u_+<0$ and \eqref{cond:U+}, there exist positive constants $C_0, \d_0$, and $\e,$ such that the following holds.
    
    For any $u_- < 0 $ satisfying $u_->u_+$ and $\delta:= |u_- - u_+| < \d_0$, let $\r_->0$ be the (unique) constant state such that $(\r_-, u_-)$ lies on the 2-Hugoniot curve. Let $(\rhotil, \util)$ be the viscous 2-shock wave connecting $(\rho_-, u_-)$ and $(\rho_+, u_+)$ with $\rhotil(0) = (\r_- + \r_+)/2$. Then, there exists a sufficiently large constant $\beta>0$ (depending only on the shock strength $|u_- - u_+|$) such that the following holds. 

    Suppose that $(\r, u)$ is the solution to the outflow problem \eqref{NS}-\eqref{BC} on $[0,T]$ for some $T>0$. Moreover, assume that
    \begin{equation*}
    \begin{aligned}
    &\r - \rhotil^{X, \beta} \in C([0,T];H^1(\mathbb{R}_+)),\\
    & u - \util^{X, \beta} \in C([0,T];H^1(\mathbb{R}_+)) \cap L^2(0,T;H^2(\mathbb{R}_+)),
    \end{aligned} 
    \end{equation*}
    and
    \begin{equation}\label{p-assum}
    \| \r - \rhotil^{X, \beta} \|_{L^\infty(0,T;H^1(\mathbb{R}_+))} + \| u - \util^{X, \beta} \|_{L^\infty(0,T;H^1(\mathbb{R}_+))} \leq \e.
    \end{equation}
    Then, for all $0\leq t \leq T$, we have
    \begin{equation}\label{est:pri}
    \begin{aligned}
    &\sup_{t \in [0,T]} \left[ \| \r - \rhotil^{X, \beta} \|_{H^1(\mathbb{R}_+)} + \| u - \util^{X, \beta} \|_{H^1(\mathbb{R}_+)}\right]\\
    &\hspace{2cm} + \sqrt{\int_0^t \left(\delta |\dot{X}|^2 + G^{\text{new}} + G^S + G^{\text{bd}}+ D_\r + D_{u1} + D_{u2}\right)\,ds}\\
    &\qquad \leq C_0 \left( \| \r_0 - \rhotil^{\beta} \|_{H^1(\mathbb{R}_+)} + \| u_0 - \util^{\beta} \|_{H^1(\mathbb{R}_+)} \right) + C_0 e^{-C\delta \beta},
    \end{aligned} 
    \end{equation}
    where $C_0$ is independent of $T$, and
    \begin{equation}\label{terms}
    \begin{split}
        &\begin{aligned}
            &G^{\text{new}} := \intRp a_x \frac{(\s - \util)}{2}\g \rhotil^{\g - 2}\left[(\r - \rhotil) - \frac{\rhotil}{(\s - \util)}(u-\util) \right]^2\,dx,
        \end{aligned}\\
        &\begin{aligned}
            &G^S := \intRp |\util_x||u- \util|^2\,dx, &
            &G^{\text{bd}} := \left.-\frac{1}{2}u_- \left(\frac{(\rho - \rhotil)_x}{\rho}\right)^2\right|_{x=0},\\
            &D_\r := \intRp \frac{p'(\rho)}{\rho}|(\r - \rhotil)_x|^2\,dx, &
            &D_{u1} := \intRp a|(u - \util)_x|^2\,dx, \quad D_{u2} := \intRp |(u - \util)_{xx}|^2\,dx.
        \end{aligned}
    \end{split}
    \end{equation}
    In addition, by \eqref{def:X},
    \[
    |\dot{X}(t)| \leq C_0 \big( \| (\r - \rhotil^{X, \beta})(t,\cdot) \|_{L^\infty(\mathbb{R}_+)} + \| (u - \util^{X, \beta})(t,\cdot) \|_{L^\infty(\mathbb{R}_+)} \big), \quad \forall t \leq T.
    \]
    \end{prop}
    Here, for any function $f:\RR \to \RR$, we use the abbreviated notation: 
    \[
f^{X,\b}(\cdot) := f(\cdot - X(t) - \b), \quad f^{\b}(\cdot) := f(\cdot - \b).
    \]
    \begin{rem}
    In Proposition~\ref{p-est}, the constant $\beta>0$ depends on the shock strength $\delta = |u_- - u_+|$. This follows from the proof of Lemma~\ref{zero}, where $\beta$ is chosen so that the quantity $Ce^{-C\delta \beta}$ in \eqref{def:y0} is less than $\frac{1}{6}$.
    \end{rem}
\subsection{Proof of Theorem \ref{thm:main}}
Based on Propositions \ref{prop:loc} and \ref{p-est}, we use the continuation argument to prove the global existence of solution. We also apply Proposition~\ref{p-est} to prove the long-time behavior \eqref{asym-U}. Since those proofs are almost identical to the proof in \cite{Kang23Inflow}, we will present them in Appendix \ref{Appenidx B}.

\section{$L^2$ energy estimates}
\setcounter{equation}{0}
To prove Proposition~\ref{p-est}, we first establish the $L^2$ estimates based on the method of $a$-contraction with shifts. To this end, we prove the following lemma.
\begin{lem}\label{zero}
    Under the hypothesis of Proposition~\ref{p-est}, there exists a positive constant $C$ such that for all $t \in [0,T]$,
\begin{equation}\label{est:zero}
    \begin{aligned}
        &\norm{\r - \rhotilX}_{L^2(\Rp)}^2 +  \norm{u - \utilX}_{L^2(\Rp)}^2 + \int_0^t (\delta |\dot{X}(s)|^2 + \Gnew + G^S + D_{u1})\,ds\\
        & \quad \leq C\left(\norm{\r_0 - \rhotil^{\b}}_{L^2(\Rp)}^2 + \norm{u_0 - \util^{\b}}_{L^2(\Rp)}^2 \right) + Ce^{-C\delta \b} + C\e^2 \int_0^t D_{u2}\,ds,
    \end{aligned}
\end{equation}
    where $\Gnew, G^S, D_{u1}$ and $D_{u2}$ are the terms defined in \eqref{terms}.
\end{lem}

First, to prove Lemma~\ref{zero}, observe from \eqref{p-assum} that
\begin{equation}\label{est:linf}
    \norm{\r - \rhotil}_{L^\infty(\Rp)}^2 +  \norm{u - \util}_{L^\infty(\Rp)}^2 \leq C\big(\norm{\r - \rhotil}_{H^1(\Rp)}^2 +  \norm{u - \util}_{H^1(\Rp)}^2\big) \leq C\e.
\end{equation}
Since the diffusion $D_{u1}$ is related to the small perturbation of $u-\util$, we perform a Taylor expansion around $\utilX$ and then apply Lemma~\ref{poincare} to show that the left-hand side of \eqref{id-rel} (excluding the boundary term $\mathcal{P}$) is non-positive .

\subsection{Relative entropy method}
In order to attain the $L^2$ estimate for the perturbation, we use the relative entropy method, which was first introduced by Dafermos and DiPerna \cite{D96, D79}.

The system \eqref{NS} can be written as a system of viscous conservation laws:
\begin{equation}\label{abs:NS}
    U_t + f(U)_x = \rd_x \left(M(U)\rd_x \nb \eta(U)\right),
\end{equation}
where the conserved quantity $U$, flux $f$, diffusion matrix $M$, and the entropy $\eta$ of the system \eqref{NS} are defined by
\[
\begin{aligned}
    &U := \begin{pmatrix}
            \rho \\ m
    \end{pmatrix} = \begin{pmatrix}
        \rho \\ \r u
\end{pmatrix}, \quad f(U) := \begin{pmatrix}
    m \\ \frac{m^2}{\r} + p(\r)
\end{pmatrix}, \quad M(U) := \begin{pmatrix}
    0 & 0 \\ 0 & 1
\end{pmatrix}.\\
    &\eta(U) = \frac{1}{2}\frac{m^2}{\r} + \frac{\r^\g}{\g - 1}.
\end{aligned}
\]
Similarly, the shifted viscous shock wave 
\[
\Util = \UtilX(t,x) := \begin{pmatrix}
\rhotilX(t,x)\\
\tilde{m}^{X,\b}(t,x)
\end{pmatrix}
\]
satisfies the following system:
\begin{equation}\label{abs:VS}
    \rd_t \Util + \rd_x f(\Util) = \rd_x \big(M(\Util)\rd_x \nb \eta(\Util)\big) - \dot{X} \rd_x \Util.
\end{equation}
Consider the relative entropy functional defined by
\[
    \eta(U|V) := \eta(U) - \eta(V) - \nb \eta(V)(U-V),
\]
and the relative flux given by 
\[
    f(U|V) := f(U) - f(V) - \nb f(V)(U - V).
\]
Let $q(U;V)$ be the flux of the relative entropy defined by
\[
    q(U;V) := q(U) - q(V) - \nb \eta(V)(f(U) - f(V)),
\]
where $q(U)$ is the entropy flux of $\eta$, i.e., $\nb q(U) = \nb \eta(U) \nb f(U)$.

For our system \eqref{abs:NS}, a straightforward computation yields
\begin{equation}\label{id-ftn}
\begin{aligned}
&\eta(U|\Util) = \r \frac{|u-\util|^2}{2} + \frac{p(\r | \rhotil)}{\g - 1},\\
&q(U;\Util) = \frac{1}{2}\r u|u-\util|^2 + \frac{u}{\g - 1} p(\r | \rhotil) + (p-\ptil)(u-\util), \\
&f(U|\Util) = \begin{pmatrix}
    0 \\ \r(u-\util)^2 + p(\r|\rhotil)
\end{pmatrix}.
\end{aligned}
\end{equation}
Here, $\ptil$ and $p(\r | \rhotil)$ are defined by $\ptil:=p(\rhotil)$ and 
\[
p(\r | \rhotil) := p(\r) - p(\rhotil) - p'(\rhotil)(\rho - \rhotil).
\]
Observe that since $p(\cdot)$ and $\eta(\cdot)$ are strictly convex on $\{\r \, | \, \r > 0\}$ and $\{U=(\r,\r u) \, | \, \r >0\}$, respectively, the relative quantities $p(\cdot|\cdot)$ and $\eta(\cdot|\cdot)$ are positive definite and locally quadratic. In particular, denoting $U_1 := (\r_1, \r_1 u_1)$ and $U_2 := (\r_2, \r_2 u_2)$, we have
\begin{equation*}
    \begin{aligned}
    &p(\rho_1|\rho_2) \geq 0, \, \,  \forall \r_1, \r_2 >0,  &&p(\rho_1|\rho_2) = 0 \iff \rho_1 = \rho_2,\\
    &\eta(U_1|U_2) \geq 0, \, \, \forall U_1, U_2 \, \, \text{with } \r_1, \r_2 >0, &&\eta(U_1|U_2) = 0 \iff U_1 = U_2.
    \end{aligned}
\end{equation*}
Moreover, for any compact sets $K_1 \subset \{\r \, | \, \r > 0\}$ and $K_2 \subset \{U=(\r,\r u) \, | \, \r > 0\}$, there exists a constant $C>0$ such that
\begin{align*}
   &C^{-1} |\r_1 - \r_2|^2 \leq p(\r_1 |\r_2) \leq C|\r_1 - \r_2|^2, \quad \forall \r_1, \r_2 \in K_1\\
   &C^{-1} |U_1 - U_2|^2 \leq \eta(U_1|U_2) \leq C|U_1 - U_2|^2, \quad \forall U_1, U_2 \in K_2. 
\end{align*}
In what follows, we compute the time evolution of the relative entropy functional, weighted by the function $a$ in \eqref{def:a}.
\begin{lem}\label{lem:rel}
    Let $a$ be the weight function defined in \eqref{def:a}, $U$ be a solution to \eqref{abs:NS}, and $\Util$ be the shifted shock wave satisfying \eqref{abs:VS}. Then,
    \begin{equation}\label{id-rel}
        \frac{d}{dt}\int_{\Rp} a(t,x) \eta(U(t,x)|\Util(t,x))\,dx = \dot{X}Y + \mathcal{J}^{\text{bad}} - \mathcal{J}^{\text{good}} + \mathcal{P},
    \end{equation}
    where
    \begin{equation*}
        \begin{aligned}
            &Y:=% \int_{\Rp} a_x \eta(U|\Util)\,dx + \int_{\Rp} a \nb^2 \eta(\Util)(U - \Util)\Util_x \,dx\\
            -\int_{\Rp} a_x \eta(U|\Util)\,dx + \int_{\Rp} a \frac{p'(\rhotil)}{\rhotil}(\r - \rhotil)\rhotil_x\,dx + \intRp a \r (u-\util)\util_x\,dx,\\
             &\mathcal{J}^{\text{bad}} := \intRp a_x(p-\ptil)(u-\util)dx + \frac{1}{2} \left(\intRp a_x \Big[\r (u-\util)^3 -(\s - \util)(\r - \rhotil)(u-\util)^2\Big] dx\right)\\ 
             &\quad \qquad + \intRp a_x (u-\util)\frac{p(\rho|\rhotil)}{\g -1} \,dx - \intRp a \r (u - \util)^2 \util_x dx - \intRp a p(\r|\rhotil) \util_x\,dx\\
            &\quad \qquad - \intRp a_x (u-\util)(u-\util)_x\,dx -\intRp a \frac{1}{\rhotil} (u-\util)(\r - \rhotil)\util_{xx}\,dx,\\
            &\mathcal{J}^{\text{good}} := \intRp a_x(\s - \util)\left[ \rhotil \frac{|u-\util|^2}{2} + \frac{p(\r|\rhotil)}{\g - 1}\right] \,dx +\intRp a |\rd_x(u-\util)|^2\,dx,\\
            &\mathcal{P} := a(t,0)q(U;\Util)(t,0) -a(t,0)(u-\util)(t,0)(u-\util)_x(t,0).
        \end{aligned}
    \end{equation*}
\end{lem}
\begin{remark}
$\mathcal{J}^{\text{good}}$ consists of good terms, due to $a_x (\s - \util) >0$ by Lemma~\ref{pos2} and \eqref{der-a}.
\end{remark}
\begin{proof}
    By the relative entropy method (see, for example, \cite{KV21}), we have
\begin{align*}
    &\frac{d}{dt}\int_{\Rp} a(t,x) \eta(U(t,x)|\Util(t,x))\,dx  = \dot{X}\left[\int_{\Rp} a_x \eta(U|\Util)\,dx + \int_{\Rp} a \rd_x(\nb \eta(\Util))(U - \Util)\,dx \right]\\
    &\quad - \s \intRp a_x \eta(U|\Util)\,dx - \intRp a \rd_x q(U;\Util)\,dx - \intRp a \rd_x \nb \eta(\Util)f(U|\Util)\,dx\\
    &\quad +\intRp a (\nb \eta(U) - \nb \eta(\Util))\rd_x \big(M \rd_x (\nb \eta(U) - \nb \eta(\Util))\big) + \intRp a (\nb \eta)(U|\Util) \rd_x \big(M \rd_x \eta (\Util)\big)\\
    &\quad =: \dot{X}Y - \s \intRp a_x \eta(U|\Util)\,dx + \sum_{i=1}^4 I_i.
\end{align*}
Here, $(\nb \eta)(U|\Util)$ is defined by $(\nb \eta)(U|\Util) := \nb \eta(U) - \nb \eta(\Util) - \nb^2 \eta(\Util)\cdot (U - \Util)$.

Using integration by parts and \eqref{id-ftn}, together with the identity
\begin{equation}\label{nbeta}
    \nb \eta(U) = \begin{pmatrix}
    -\frac{u^2}{2} + \frac{\gamma}{\gamma - 1}\r^{\gamma -1} \\ u
\end{pmatrix},
\end{equation}
we have
\begin{equation*}
    \begin{aligned}
        I_1 &=  \intRp a_x q(U;\Util)\,dx + a(t,0)q(U;\Util)(t,0),\\
        I_2 &= - \intRp a \util_x (\r (u - \util)^2 + p(\r|\rhotil))\,dx.
    \end{aligned}
\end{equation*}
For the parabolic terms $I_3$ and $I_4$, using \eqref{nbeta} and the fact that the second component of $(\nb \eta)(U|\Util)$ is $-\frac{1}{\rhotil}(\rho - \rhotil)(u-\util)$, we obtain
\begin{equation*}
    \begin{aligned}
    I_3 &= -\intRp a |\rd_x(u-\util)|^2\,dx - \intRp a_x (u-\util)(u-\util)_x\,dx\\
    &\quad -a(t,0)(u-\util)(t,0)(u-\util)_x(t,0),\\
    I_4 &= -\intRp a \frac{1}{\rhotil} (u-\util)(\r - \rhotil)\util_{xx}\,dx.
\end{aligned}
\end{equation*}
Thus, we have
\begin{align*}
    &\frac{d}{dt} \intRp a \eta(U|\Util)\,dx\\
    &\quad = \dot{X}Y - \s \intRp a_x \left[\r \frac{|u-\util|^2}{2} + \frac{p(\r | \rhotil)}{\g - 1}\right]\,dx \\
    &\qquad + \frac{1}{2} \intRp a_x \r u(u-\util)^2\,dx+\intRp a_x \frac{u}{\g - 1}p(\r|\rhotil)\,dx + \intRp a_x (p-\ptil)(u-\util)\,dx\\
    &\qquad - \intRp a  (\r (u - \util)^2 + p(\r|\rhotil))\util_x \,dx- \intRp a_x (u-\util)(u-\util)_x\,dx\\
    &\qquad -\intRp a \frac{1}{\rhotil} (u-\util)(\r - \rhotil)\util_{xx}\,dx-\intRp a |\rd_x(u-\util)|^2\,dx +\mathcal{P}.
    %&\qquad + a(t,0)q(U;\Util)(t,0) -a(t,0)(u-\util)(t,0)(u-\util)_x(t,0).\\
\end{align*}
Equivalently, we obtain
\begin{align*}
    &\frac{d}{dt} \intRp a \eta(U|\Util)\,dx\\
    &\, = \dot{X}Y - \intRp a_x(\s - \util) \frac{p(\r | \rhotil)}{\g - 1}\,dx + \intRp a_x (p-\ptil)(u-\util)\,dx\\
    \
    &\, + \intRp a_x (u-\util)\frac{p(\r|\rhotil)}{\g - 1}\,dx + \frac{1}{2} \left(\intRp a_x \Big[\r (u-\util)^3 -(\s - \util)(\r - \rhotil)(u-\util)^2\Big] dx\right)\\
    &\,- \intRp a (\r (u - \util)^2 + p(\r|\rhotil))\util_x\,dx  - \intRp a_x (u-\util)(u-\util)_x\,dx\\
    &\, -\intRp a \frac{1}{\rhotil} (u-\util)(\r - \rhotil)\util_{xx}\,dx - \intRp a_x(\s - \util)\rhotil \frac{|u-\util|^2}{2} \,dx -\intRp a |\rd_x(u-\util)|^2\,dx\   + \mathcal{P}\\
    % + a(t,0)q(U;\Util)(t,0) -a(t,0)(u-\util)(t,0)(u-\util)_x(t,0)\\
    &\, = \dot{X}Y + \mathcal{J}^{\text{bad}} - \mathcal{J}^{\text{good}} + \mathcal{P}.
\end{align*}
\end{proof}
\subsection{Maximization in terms of $\r - \rhotil$}
On the right-hand side of \eqref{id-rel}, we use Lemma~\ref{poincare} to control the bad terms only related to the perturbation $u - \util$. Therefore, we rewrite the worst term in $\mathcal{J}^{\text{bad}}$%:
%\[
%    \intRp a_x (p(\r) -p(\rhotil))(u-\util)\,dx
%\]
in terms of $\r - \rhotil$ in the following lemma. We remark that in $\mathcal{J}^{\text{bad}}$, the term 
\[
  \intRp a_x (p(\r) -p(\rhotil))(u-\util)\,dx
\]
is the worst term, as the scale of $a_x$ is larger than that of $\util_x$ (see \eqref{der-a}), and it is a crossing term involving both $u-\util$ and $\r - \rhotil$.
\begin{lem}\label{lem:max}
    Under the hypothesis of Proposition~\ref{p-est}, we have
\begin{align*}
\begin{aligned}
&-\intRp a_x (\s - \util) \frac{p(\r|\rhotil)}{\g - 1} +\intRp a_x (p-\ptil)(u-\util)\,dx\\
&\quad \leq -\intRp a_x\frac{(\s - \util)}{2}\g \rhotil^{\g - 2}\left[(\r - \rhotil) - \frac{\rhotil}{(\s - \util)}(u-\util) \right]^2\,dx + \intRp a_x\frac{\g \rhotil^{\g}}{2(\s - \util)}(u-\util)^2\,dx \\
    &\qquad +C\intRp a_x |\rho - \rhotil|^2|u-\util|\,dx + C\intRp a_x (\s - \util) |\r - \rhotil|^3 \, dx.
\end{aligned}
\end{align*}
\end{lem}
\begin{proof}
    Using \eqref{est:linf} and Taylor theorem, we obtain
\begin{equation}\label{exp-p}
\begin{aligned}
    &\left|p(\r|\rhotil) - \frac{\g (\g -1)}{2}\rhotil^{\g - 2}(\r - \rhotil)^2 \right| \leq C|\r - \rhotil|^3,\\
    &\Big|p(\rho) - p(\rhotil) - \g \rhotil^{\g - 1}(\r - \rhotil)\Big| \leq C|\r - \rhotil|^2.
\end{aligned}
\end{equation}
Using \eqref{exp-p}, we have
\begin{align*}
    & -\intRp a_x (\s - \util) \frac{p(\r|\rhotil)}{\g - 1}\,dx + \intRp a_x (p-\ptil)(u-\util)\,dx\\
    &\quad \leq - \intRp  a_x(\s - \util) \frac{\g}{2}\rhotil^{\g -2}(\r - \rhotil)^2 \,dx + C\intRp a_x (\s - \util) |\r - \rhotil|^3 \, dx\\
    &\qquad +\intRp a_x \g \rhotil^{\g - 1}(\r - \rhotil)(u-\util)\,dx +C\intRp a_x |\rho - \rhotil|^2|u-\util|\,dx\\
    &\quad = - \intRp  a_x(\s - \util) \frac{\g}{2}\rhotil^{\g -2}(\r - \rhotil)^2 \,dx + \intRp a_x \g \rhotil^{\g - 1}(\r - \rhotil)(u-\util)\,dx\\
    &\qquad +C\intRp a_x |\rho - \rhotil|^2|u-\util|\,dx + C\intRp a_x (\s - \util) |\r - \rhotil|^3 \, dx.
\end{align*}
Applying the quadratic identity $a x^2 + b x = a\left(x+ \frac{b}{2a}\right)^2 - \frac{b^2}{4a}$ with $x := \r - \rhotil$, we obtain
\begin{align*}
    &-\intRp a_x (\s - \util) \frac{p(\r|\rhotil)}{\g - 1} \,dx  +\intRp a_x (p-\ptil)(u-\util)\,dx\\
    &\quad \leq -\intRp a_x \frac{(\s - \util)}{2}\g \rhotil^{\g - 2}\left[(\r - \rhotil) - \frac{\rhotil}{(\s - \util)}(u-\util) \right]^2\,dx + \intRp a_x \frac{\g \rhotil^{\g}}{2(\s - \util)}(u-\util)^2\,dx \\
    &\qquad +C\intRp a_x |\rho - \rhotil|^2|u-\util|\,dx + C\intRp a_x (\s - \util) |\r - \rhotil|^3 \, dx.
\end{align*}
\end{proof}
Using Lemma~\ref{lem:rel}, Lemma~\ref{lem:max} and \eqref{exp-p}, we have the following lemma.
\begin{lem}\label{lem:XBG}
    Under the hypothesis of Proposition~\ref{p-est}, we have
    \[
    \frac{d}{dt} \intRp a \eta(U|\Util)\,dx \leq \dot{X}Y + B - G+ \mathcal{P},
    \]
    where $Y$ and $\mathcal{P}$ are defined in Lemma~\ref{lem:rel}, and
    \[
     B = \sum_{i=1}^6 B_i, \quad  G = \Gnew + G_1 + D_{u1}, 
    \]
where 
\begin{equation}\label{termsBG}
    \begin{split}
    \begin{aligned}
    &B_1 := \intRp a_x \frac{\g \rhotil^{\g}}{2(\s - \util)}(u-\util)^2\,dx, && B_2 := - \intRp a \r (u - \util)^2 \util_x \,dx,\\
    &B_3 := - \intRp a \frac{\g (\g-1)}{2}\rhotil^{\g - 2}(\r - \rhotil)^2 \util_x  \, dx, && B_4:= - \intRp a_x (u-\util)(u-\util)_x\,dx,\\
    &B_5 := -\intRp a \frac{1}{\rhotil} (u-\util)(\r - \rhotil)\util_{xx}\,dx, && B_6:= C \intRp a_x \Big[|\rho - \rhotil| + |u-\util| \Big]^3\,dx,\\
    &G_1 := \frac{1}{2}\intRp a_x(\s - \util)\rhotil|u-\util|^2\,dx,
    \end{aligned}
    \end{split}
    \end{equation}
and $\Gnew, D_{u1}$ are as in \eqref{terms}.    
\end{lem}

\subsection{Estimates for $ \dot{X}Y + B - G$}
Here, we present the estimates of $\dot{X}Y + B - G$, defined in Lemma~\ref{lem:XBG}.
\begin{lem}\label{lem:lead}
    Under the hypothesis of Proposition~\ref{p-est}, we have
\begin{align*}
   \dot{X}Y + B - G \leq -\frac{\d}{4M}|\dot{X}|^2 - \frac{1}{2}\Gnew -\frac{\g+1}{8}\r_+G^S - \frac{1}{4}D_{u1},
\end{align*}
where $Y, B, G, \Gnew, D_{u1}$ are defined in Lemma~\ref{lem:XBG}, and $G^S$ is defined in \eqref{terms}. 
\end{lem}
\begin{proof}
First, we introduce a change of variable $x \mapsto y$ as
\begin{equation}\label{def:y}
    y := \frac{u_- - \util(x -\s t - X(t) - \b)}{\d}.
\end{equation}
Since $\util_x < 0$, the map $x \mapsto y=y(x,t)$ is well-defined, and it satisfies
\[
\frac{dy}{dx} = -\frac{\util_x(x-\s t - X(t) -\b)}{\delta} >0, \quad \lim_{x \to 0} y = y_0(t), \quad \lim_{x \to +\infty} y = 1,
\]
where
\begin{equation*}%\label{def:y00}
    y_0(t) := -\frac{u_- - \util(-\s t - X(t) - \b)}{\delta} >0.
\end{equation*}
Note from \eqref{def:X} that 
    \[
    |\dot{X}| \leq \frac{C}{\d}\left|\intRp (U-\Util)\cdot \Util_x \,dx \right|.
\]
By a priori assumption \eqref{p-assum}, we have
\[
    |\dot{X}(t)| \leq C \|U-\Util\|_{L^\infty(\Rp)}\frac{1}{\d}\intRp |\Util_x|\,dx \leq C\e,
\]
which implies 
\[
    |X(t)| \leq C\e t.
\]
If we choose $\e$ sufficiently small, we obtain
\[
    |X(t)| \leq \frac{\s}{2}t, \quad  ^\forall t\leq T.
\]
Here, we used the fact that the shock speed $\s$ is positive. This implies
\begin{equation}\label{small-X}
    -\s t - X(t) - \b \leq -\frac{\s}{2}t - \b < 0, \quad ^\forall t \leq T.
\end{equation}
Thus, we have
\begin{equation}\label{def:y0}
    y_0(t) \leq Ce^{-C\d|\s t + X(t) + \beta|} \leq C e^{-C\d \b}.
\end{equation}
Therefore, $y_0$ can be chosen arbitrary small by taking sufficiently large $\b$. In particular, we choose $\b$ so that $y_0 < \frac{1}{6}$.

For simplicity, we now rewrite the functionals $Y, B$, and $G$ with respect to the following variables
\[
w := (u - \utilX)\circ y^{-1}.
\]
$\bullet$ (Estimates of $B_1 - G_1$) 
Recall from \eqref{termsBG} that the terms $B_1$ and $G_1$ are defined by 
\[
B_1 := \intRp a_x \frac{\g \rhotil^{\g}}{2(\s - \util)}(u-\util)^2\,dx, \quad G_1 := \frac{1}{2}\intRp a_x(\s - \util)\rhotil|u-\util|^2\,dx.
\]
Then, from \eqref{pos2}, we obtain
\begin{align*}
    B_1 - G_1 \leq  \left|\intRp a_x \frac{(\s - \util)\rhotil}{2}\left[\frac{\g \rhotil^{\g - 1}}{(\s -\util)^2} - 1 \right](u-\util)^2\,dx \right| \leq C\delta \sqrt{\d} \int_{y_0}^1 w^2\,dy.
\end{align*}
$\bullet$ (Estimates of $B_2$) Now, we estimate the leading order terms $B_2$ and $B_3$. First, for the $B_2$ term, we use $1 \leq a \leq 1+C\sqrt{\delta}$ and $|\rho - \rho_+| \leq |\rho - \rhotil| + |\rhotil - \rho_+| \leq C(\e + \d)$ to have
\begin{align*}
B_2 = -\intRp a \r (u-\util)^2 \util_x \,dx \leq \delta(\r_+ + C \e + C \sqrt{\d}) \int_{y_0}^1 w^2\,dy.
\end{align*}
$\bullet$ (Estimates of $B_3$)
When we control $B_3$ by $\Gnew$, we get one more leading order term, as follows: using the algebraic inequality:
\[
(a+b)^2 \leq (1+\d^{-1/4})a^2 + (1+\d^{1/4})b^2, \quad \forall a, b \in \mathbb{R},
\]
we have
\begin{align*}
    |B_3| &= \left|-\intRp a  \frac{\g (\g-1)}{2}\rhotil^{\g - 2}(\r - \rhotil)^2 \util_x \, dx\right|\\
    &\leq \intRp a \frac{\g (\g-1)}{2}\rhotil^{\g - 2}\left|\left(\r -\rhotil - \frac{\rhotil}{(\s - \util)}(u-\util)\right) +  \left(\frac{\rhotil}{(\s - \util)}(u-\util)\right)\right|^2|\util_x|  \, dx\\
    &\leq (1+\d^{-1/4})\intRp \left(\r -\rhotil - \frac{\rhotil}{(\s - \util)}(u-\util) \right)^2 |\util_x|\,dx\\
    &\qquad  + (1+\d^{1/4})\intRp a \left(\frac{\g (\g-1)}{2}\frac{\rhotil^{\g}}{(\s-\util)^2}\right)(u-\util)^2 |\util_x| \,dx.
\end{align*}
From \eqref{der-a} and \eqref{pos2}, we obtain
\begin{align*}
    |B_3| &\leq C\d^{1/4} \Gnew +\left(\frac{\g -1}{2}\r_+ + C\d^{1/4}  + C\d^{1/2} \right)\intRp (u-\util)^2|\util_x|\,dx\\
    &=C\d^{1/4} \Gnew + \left(\frac{\g -1}{2}\r_+ + C\d^{1/4} + C\d^{1/2} \right)\d \int_{y_0}^1 w^2\,dy.
\end{align*}

$\bullet$ (Estimates of $B_4$)
Using Young's inequality and the bound $\|a_x\|_{L^\infty}\leq C\delta \sqrt{\d}$ obtained from \eqref{der-a}, we have
\begin{align*}
    |B_4| &= \left|-\intRp a_x (u-\util)(u-\util)_x \,dx\right| \leq \frac{1}{40}D_{u1} + C\intRp |a_x|^2 (u-\util)^2\,dx\\
    &\leq \frac{1}{40}D_{u1} + C\d^2 \int_{y_0}^1 w^2\,dy.
\end{align*}

$\bullet$ (Estimates of $B_5$)
From $|\util_{xx}| \leq C\d |\util_{x}| \leq C\d \sqrt{\d} |a_x|$, we get
\[
    \begin{aligned}
        |B_5| &= \left|-\intRp a \frac{1}{\rhotil} (u-\util)(\r - \rhotil)\util_{xx}\,dx\right| \\
        &\leq C\d \sqrt{\d}\left|\intRp a_x (u-\util)\left(\left(\r -\rhotil - \frac{\rhotil}{(\s - \util)}(u-\util)\right) +  \left(\frac{\rhotil}{(\s - \util)}(u-\util)\right)\right)\right|\,dx\\
        &\leq C\d \sqrt{\d}\left|\intRp a_x (u-\util)\left(\r -\rhotil - \frac{\rhotil}{(\s - \util)}(u-\util) \right)\,dx \right| +  C\d \sqrt{\d}\intRp a_x(u-\util)^2\,dx
    \end{aligned}
\]
Using Young's inequality, we obtain 
\begin{align*}
    |B_5| &\leq C\d \sqrt{\d}\intRp a_x (u-\util)^2\,dx +C\d \sqrt{\d} \intRp a_x \left(\r -\rhotil - \frac{\rhotil}{(\s - \util)}(u-\util) \right)^2\,dx\\
    &\qquad  +  C\d \sqrt{\d}\intRp a_x(u-\util)^2\,dx\\
    &\leq  \frac{1}{40}\Gnew + C\d^2 \int_{y_0}^1 w^2\,dy. 
\end{align*}
$\bullet$ (Estimates of $B_6$)
Recall from \eqref{termsBG} that the term $B_6$ is defined by 
\[
B_6 = C \intRp a_x \Big[|\rho - \rhotil| + |u-\util| \Big]^3\,dx.
\]
For the cubic bad term $B_6$, we use the inequality $(|a|+|b|)^3 \leq C(|a-b|^3 + |b|^3)$ and the interpolation inequality to have 
\[
    \begin{aligned}
        |B_6| &\leq C\intRp a_x \left|\r -\rhotil - \frac{\rhotil}{(\s - \util)}(u-\util) \right|^3\,dx + C\intRp a_x|u-\util|^3\,dx\\
        &\leq C\e \Gnew + C\|u-\util\|_{L^\infty(\Rp)}^2 \intRp a_x |u-\util|\,dx\\
        &\leq  C\e \Gnew + C\|(u-\util)_x\|_{L^2(\Rp)}\|(u-\util)\|_{L^2(\Rp)} \intRp a_x |u-\util|\,dx.
    \end{aligned}
\]
Using a priori assumption \eqref{p-assum},H\"older's inequality, and Young's inequality, we get 
\[
    \begin{aligned}
        |B_6| &\leq  C\e \Gnew +C\e\|(u-\util)_x\|_{L^2(\Rp)}\sqrt{\intRp |a_x|\,dx}\sqrt{\intRp |a_x| (u-\util)^2\,dx}\\
        &\leq C\e \Gnew +C\e\|(u-\util)_x\|_{L^2(\Rp)}^2 + C\e \intRp |a_x|\,dx \intRp |a_x| (u-\util)^2\,dx\\
        &\leq  C\e \Gnew + C\e D_{u1} + C \e \d  \int_{y_0}^1 w^2\,dy.
    \end{aligned}
\]
\noindent $\bullet$ (Estimates of $D_{u1}$)
To estimate the diffusion, we need to approximate the Jacobian $\frac{dy}{dx}$ for the change of variable. This is provided by the following lemma.
\begin{lem}\label{lem:jac}
    There exists a constant $C>0$ such that
\begin{equation*}
    \left|\frac{1}{y(1-y)}\frac{dy}{dx} - \frac{\g + 1}{2}\r_+ \d\right| \leq C\d^2.
\end{equation*}
\end{lem}
\begin{proof}
First, observe from \eqref{def:y} that  
    \begin{align*}
        \frac{1}{y(1-y)}\frac{dy}{dx} = \left(\frac{1}{y} + \frac{1}{1-y}\right)\frac{dy}{dx} = \frac{\util'}{\util - u_-} - \frac{\util'}{\util - u_+}.
    \end{align*}
    Integrating \eqref{Vshock} over $(-\infty, \xi)$ and $(\xi ,\infty)$, we have 
    \begin{equation}\label{int_VS}
        \rhotil(\util-\s) = \r_\pm (u_\pm - \s),
    \end{equation}
    and then,
    \begin{align*}
        \util' = \util \rhotil(\util - \s) - \r_\pm u_\pm ( u_\pm - \s) + \ptil - p_\pm = \r_\pm (\util - u_\pm)(u_\pm - \s) + \ptil - p_\pm.
         %&= \util \rhotil(\util - \s) - \r_- u_- ( u_- - \s) + \ptil - p_-\\
        %&= \r_- (\util - u_-)(u_- - \s) + \ptil - p_-\\
        %&= \r_+ (\util - u_+)(u_+ - \s) + \ptil - p_+.
    \end{align*}
    From \eqref{RH} and above identities, we get 
    \begin{equation}\label{est:Jac}
        \begin{aligned}
            \frac{\util'}{\util - u_-} - \frac{\util'}{\util - u_+} &= \r_-(u_- - \s) - \r_+(u_+ - \s) + \frac{\ptil - p_-}{\util - u_-} - \frac{\ptil - p_+}{\util - u_+}\\
            &= \frac{\ptil - p_-}{\util - u_-} - \frac{\ptil - p_+}{\util - u_+}.
        \end{aligned}
    \end{equation}
    Moreover, \eqref{int_VS} implies that 
    \[
    \util - u_- = (\s - u_-)\frac{\rhotil - \r_-}{\rhotil}, \quad \text{and} \quad \util - u_+ = (\s - u_+)\frac{\rhotil - \r_+}{\rhotil}.
    \]
    Thus, we have 
    \begin{align*}
        \frac{\ptil - p_-}{\util - u_-} - \frac{\ptil - p_+}{\util - u_+} &= \rhotil \left[ \frac{\ptil - p_-}{(\s - u_-)(\rhotil - \r_-)} -  \frac{\ptil - p_+}{(\s - u_+)(\rhotil - \r_+)}\right]\\
        &=\frac{\rhotil}{(\s - u_-)(\s - u_+)}\underbrace{\left[(\s - u_+)\frac{\ptil - p_-}{\rhotil - \r_-} - (\s - u_-)\frac{\ptil - p_+}{\rhotil - \r_+} \right]}_{=:J}.
    \end{align*}
    Using \eqref{pos2}, we find that
    \begin{equation}\label{est:Jac2}
        \frac{\ptil - p_-}{\util - u_-} - \frac{\ptil - p_+}{\util - u_+} = \left(\frac{\r_+}{p'(\r_+)} + O(\d)\right)J.
    \end{equation}
    %Now, we use the identity(See \cite[Chapter~11]{Evans}):
    %\begin{align*}
    %    \s &= \frac{\lambda_2(U_-)+ \lambda_2(U_+)}{2} + O(\d^2)\\
    %    &= \frac{u_- + u_+}{2} + \frac{1}{2}\left(\sqrt{p'(\r_+)} + \sqrt{p'(\r_-)}\right) + O(\d^2)
    %\end{align*}
    %to have
    Therefore, it remains to estimate $J$. To this end, note that 
    \begin{align*}
        J &= (\s - u_-)\left[\frac{\ptil - p_-}{\rhotil - \r_-} - \frac{\ptil - p_+}{\rhotil - \r_+} \right] + (u_+ - u_-)\frac{\ptil - p_-}{\rhotil - \r_-}\\
        &=(\s - u_-)\left[\frac{\ptil - p_-}{\rhotil - \r_-} - \frac{\ptil - p_+}{\rhotil - \r_+} \right] + p'(\r_+)\d + O(\d^2).
    \end{align*}
    Moreover, observe that 
    \begin{align*}
        &\Big|\frac{\ptil - p_-}{\rhotil - \r_-} - \frac{\ptil - p_+}{\rhotil - \r_+} - \frac{p''(\r_+)}{2}(\r_- - \r_+)\Big|\\
        &\, \, \leq \Big| p'(\r_-) + \frac{p''(\r_-)}{2}(\rhotil - \r_-) - p'(\r_+) - \frac{p''(\r_+)}{2}(\rhotil - \r_+) - \frac{p''(\r_+)}{2}(\r_- - \r_+)\Big| + C\d^2\\
        &\,\,  =\Big|p'(\r_-) - p'(\r_+) + \underbrace{\frac{p''(\r_-)}{2}(\rhotil - \r_-) - \frac{p''(\r_+)}{2}(\rhotil - \r_-)}_{\leq C\d^2} - p''(\r_+)(\r_- - \r_+)\Big|+ C\d^2\\
        &\leq \,\, \big|p'(\r_-) - p'(\r_+) - p''(\r_+)(\r_- - \r_+)\big|+ C\d^2\\
        &\, \,  \leq C\d^2.
    \end{align*}
    Using \eqref{RH} and \eqref{pos2}, we have
    \[
    \r_- - \r_+ = \frac{\r_+(u_- - u_+)}{\s - u_-} = \frac{\r_+}{\s - u_-}\d.
    \]
  This implies that
    \[
        \left|\frac{\ptil - p_-}{\rhotil - \r_-} - \frac{\ptil - p_+}{\rhotil - \r_+} - \frac{p''(\r_+)}{2}\frac{\r_+}{\s - u_-}\d \right| \leq C\d^2.
    \]
    Thus, we obtain 
    \begin{equation}\label{est:J}
        \begin{aligned}
            J &= (\s - u_-)\left[\frac{\ptil - p_-}{\rhotil - \r_-} - \frac{\ptil - p_+}{\rhotil - \r_+} \right] + p'(\r_+)\d + O(\d^2)\\
            &= \left[\frac{1}{2}p''(\r_+)\r_+ +p'(\r_+) \right]\d + O(\d^2) = \frac{\g +1}{2}\g \r_+^{\g -1}\d + O(\d^2)\\
            &= \frac{\g+1}{2}p'(\r_+) \d + O(\d^2).
        \end{aligned}
    \end{equation}
    Hence, from \eqref{est:Jac}, \eqref{est:Jac2} and \eqref{est:J}, we conclude that 
    \begin{align*}
        \frac{1}{y(1-y)}\frac{dy}{dx} &= \left(\frac{\rho_+}{p'(\r_+)} + O(\d)\right)\left(\frac{\g+1}{2}p'(\r_+) \d + O(\d^2) \right)\\
        &=\frac{\g + 1}{2}\r_+ \d + O(\d^2).
    \end{align*}
\end{proof}
From Lemma~\ref{lem:jac} and $1\leq a \leq 1+C\sqrt{\d}$, we obtain 
\begin{align*}
    D_{u1} &= \intRp a |\rd_y(u-\util)|^2\frac{dy}{dx}\,dy \geq \left(\frac{\g+1}{2}\r_+\d - C\d^2 \right)\int_{y_0}^1 y(1-y)|\rd_y w|^2\,dy\\
&\geq \left(\frac{\g+1}{2}\r_+\d- C\d^2 \right)\int_{y_0}^1 (y-y_0)(1-y)|\rd_y w|^2\,dy.
\end{align*}
$\bullet$ (Estimates of $Y$)
For $Y$, we decompose the functional $Y$ as follows:
\[
    Y := Y_g + Y_1 + Y_2,
\]
where
\begin{align*}
    &Y_g := \intRp a \frac{p'(\rhotil)}{(\s -\util)}(u-\util)\rhotil_x\,dx + \intRp a \rhotil (u - \util)\util_x\,dx,\\
    &Y_1 :=  -\intRp a_x \eta(U|\Util)\,dx = - \intRp a_x \r \frac{(u-\util)^2}{2}\,dx - \intRp a_x \frac{p(\r|\rhotil)}{\g - 1}\,dx,\\
    &Y_2 := \intRp a \frac{p'(\rhotil)}{\rhotil}\left(\r -\rhotil - \frac{\rhotil}{(\s - \util)}(u - \util)\right)\rhotil_x \, dx + \intRp a (\r - \rhotil)(u-\util)\util_x \, dx.    
\end{align*}
Here, $Y_g$ is the term that contributes the sharp estimate based on the inequality of Lemma~\ref{poincare}.
Notice from \eqref{def:X} that 
\[
    \dot{X}(t) = -\frac{M}{\d}Y_g,
\]
and so,
\begin{align*}
    \dot{X}Y &= -\frac{\d}{M}|\dot{X}(t)|^2 + \dot{X}(t)Y_1 +  \dot{X}(t)Y_2 \\
    &\leq -\frac{\d}{2M}|\dot{X}(t)|^2 + \frac{C}{\d}\big(|Y_1|^2 + |Y_2|^2 \big)= -\frac{M}{2\d}|Y_g|^2 + \frac{C}{\d}\big(|Y_1|^2 + |Y_2|^2 \big).
\end{align*}
First of all, using \eqref{Vshock}, we obtain 
\[
-(\s - \util)\rhotil \util' + p'(\rhotil)\rhotil' = \util\,''.
\]
Using the above identity, we have
\[
   \begin{aligned}
        Y_g &= \intRp a \frac{p'(\rhotil)}{(\s -\util)}(u-\util)\rhotil_x\,dx + \intRp a \rhotil (u - \util)\util_x\,dx\\
        &= 2\int a \rhotil (u-\util)\util_x \,dx + \intRp \frac{a}{(\s - \util)}(u-\util)\util_{xx}\,dx.
   \end{aligned} 
\]
This identity, together with the estimate $|\util_{xx}|\leq C\d |\util_x|$, implies
\[
\left| Y_g + 2\r_+ 
\d \int_{y_0}^1 w\,dy \right| \leq  C\d (\sqrt{\d}+\d)\int_{y_0}^1 |w|\,dy.
\]
Using the algebraic inequality: $a^2/2 - b^2 \leq (b-a)^2$ for $a,b\geq0$, we get
\begin{align*}
      -\frac{M}{2\d}|Y_g|^2 &\leq  -\d M\r_+^2\left(\int_{y_0}^1 w\,dy\right)^2 + C\d^2\left(\int_{y_0}^1 |w|\,dy\right)^2\\
      %&\leq -\d M\r_+^2\left(\int_{y_0}^1 w\,dy\right)^2 + C(1-y_0) \d^2\int_{y_0}^1 |w|^2\,dy\\
      &\leq -\d M\r_+^2\left(\int_{y_0}^1 w\,dy\right)^2 + C\d^2\int_{y_0}^1 |w|^2\,dy.
\end{align*}
Now, we estimate the term $Y_1 = - \intRp a_x \r \frac{(u-\util)^2}{2}\,dx - \intRp a_x \frac{p(\r|\rhotil)}{\g - 1}\,dx$. Observe that
\begin{align*}
    \intRp a_x \r \frac{(u-\util)^2}{2}\,dx &\leq C \frac{1}{\sqrt{\d}} \intRp  (u-\util)^2|\util_x|\,dx \\
    &\leq C \frac{1}{\sqrt{\d}}\|\util_x\|_{L^\infty}^{1/2} \sqrt{\intRp (u-\util)^2 \,dx}\sqrt{\intRp |\util_x|(u-\util)^2\,dx}\\
    &\leq C \sqrt{\d} \e \sqrt{\intRp |\util_x| (u-\util)^2\,dx}.
\end{align*}
Similar calculations yield that 
\[
    \intRp a_x \frac{p(\r|\rhotil)}{\g - 1}\,dx \leq C\intRp a_x (\r - \rhotil)^2\,dx \leq C \d^{3/4} \e  \sqrt{\intRp a_x (\r -\rhotil)^2\,dx}.
\]
\begin{comment}
and
\begin{align*}
    \intRp a_x \frac{p(\r|\rhotil)}{\g - 1}\,dx &\leq C\intRp a_x (\r - \rhotil)^2\,dx \\
    &\leq C \|a_x\|_\infty^{1/2}\sqrt{\intRp (\r -\rhotil)^2 \,dx} \sqrt{\intRp a_x (\r -\rhotil)^2\,dx}\\
    &\leq C(\l \d)^{1/2} \e  \sqrt{\intRp a_x (\r -\rhotil)^2\,dx},
\end{align*}
\end{comment}
Therefore, we have
\begin{align*}
    \frac{C}{\d}|Y_1|^2 &\leq \frac{C}{\d}\left(\intRp a_x \r \frac{(u-\util)^2}{2}\,dx\right)^2 + \frac{C}{\d}\left(\intRp a_x \frac{p(\r|\rhotil)}{\g - 1}\,dx \right)^2\\
    &\leq C\e^2\intRp |\util_x| (u-\util)^2\,dx + C\sqrt{\d} \e^2 \intRp a_x (\r -\rhotil)^2\,dx\\
    &\leq C\e^2\intRp |\util_x| (u-\util)^2\,dx + C\sqrt{\d} \e^2 \Gnew + C\sqrt{\d} \e^2 \intRp a_x (u-\util)^2\,dx\\
    &\leq C\sqrt{\d} \e^2 \Gnew + C\d\e^2 \int_{y_0}^1 w^2\,dy.
\end{align*}
For the term $Y_2$, note that
\begin{align*}
    |Y_2| &= \left|\intRp a \frac{p'(\rhotil)}{\rhotil}\left(\r -\rhotil - \frac{\rhotil}{(\s - \util)}(u - \util)\right)\rhotil_x \, dx + \intRp a (\r - \rhotil)(u-\util)\util_x \, dx\right|\\
    &\leq  C\intRp a \left|\frac{p'(\rhotil)}{\rhotil}\left(\r -\rhotil - \frac{\rhotil}{(\s - \util)}(u - \util)\right)\util_x\right|\,dx\\
    &\quad \quad + \intRp a\left(\left(\r -\rhotil - \frac{\rhotil}{(\s - \util)}(u - \util)\right) + \left(\frac{\rhotil}{(\s - \util)}(u - \util)\right)\right)(u-\util)|\util_x|\,dx.
\end{align*}
This implies
\begin{align*}
    |Y_2|&\leq C\Big(1 + \|u-\util\|_{L^\infty(\Rp)}\Big)\intRp a \left|\frac{p'(\rhotil)}{\rhotil}\left(\r -\rhotil - \frac{\rhotil}{(\s - \util)}(u - \util)\right)\util_x\right|\,dx\\
    &\qquad + C\intRp (u-\util)^2|\util_x|\,dx.\\
\end{align*}
Using \eqref{est:linf} and H\"older's inequality, we have
\begin{align*}
    |Y_2| &\leq C \sqrt{\d} \left|\intRp a_x \left(\r - \rhotil - \frac{\rhotil}{(\s - \util)}(u-\util)\right)\,dx \right|+ C\intRp (u-\util)^2|\util_x|\,dx\\
    &\leq C\sqrt{\d} \sqrt{\intRp |a_x| \,dx}\sqrt{\Gnew} + C\|\util_x\|_{L^\infty}^{1/2}\sqrt{\intRp (u-\util)^2\,dx}\sqrt{\intRp (u-\util)^2|\util_x|\,dx}\\
    &\leq C\d^{3/4}\sqrt{\Gnew} + C\d\e\sqrt{\intRp (u-\util)^2|\util_x|\,dx}.
\end{align*}
Thus, we obtain
\begin{align*}
    \frac{C}{\d}|Y_2|^2 %&\leq C\frac{\d}{\l}\Gnew + C\d\e^2\intRp |\util_x|(u-\util)^2\,dx\\
    &\leq C\sqrt{\d}\Gnew + C\d^2 \e^2\int_{y_0}^1 w^2\,dy.
\end{align*}
$\bullet$ (Conclusion) Combining above estimates, we get 
\begin{align*}
    &-\frac{\d}{4M}|\dot{X}|^2 + \frac{C}{\d}\big(|Y_1|^2 + |Y_2|^2 \big)\ + B - \frac{1}{2}\Gnew - G_1 - \frac{3}{4}D_{u1}\\
    &\quad \leq \delta \left[-\frac{1}{2}M\r_+^2\left(\int_{y_0}^1 w\,dy\right)^2 + \left(\frac{\g + 1}{2}\r_+ +  C\d^{1/4} + C\e\right) \int_{y_0}^1 w^2\,dy \right.\\
    &\quad \quad \quad \left. -\frac{3}{4} \left(\frac{\g+1}{2}\r_+- C\d \right)\int_{y_0}^1 (y-y_0)(1-y)|\rd_y w|^2\,dy\right].
\end{align*}
Using Lemma~\ref{poincare}:
\begin{equation*}
    \begin{aligned}
        \int_{y_0}^1 |w|^2\,dy - \frac{1}{1-y_0}\left(\int_{y_0}^1 w \,dy \right)^2 &\leq \frac{1}{2}\int_{y_0}^1 (y-y_0)(1-y)|\rd_y w|^2\,dy, 
    \end{aligned}
    \end{equation*}
with $M = \frac{2(\g+1)}{\r_+}$ and $y_0 < \frac{1}{6}$, we have
\begin{align*}
    &-\frac{\d}{4M}|\dot{X}|^2 + \frac{C}{\d}\big(|Y_1|^2 + |Y_2|^2 \big)\ + B - \frac{1}{2}\Gnew - G_1 - \frac{3}{4}D_{u1}\\
    &\quad \leq \d  \left[-\frac{1}{2}M\r_+^2\left(\int_{y_0}^1 w\,dy\right)^2 + \frac{9}{8} \frac{\g + 1}{2}\r_+ \int_{y_0}^1 w^2\,dy \right.\\
    &\phantom{quad \leq \d  -\frac{1}{2}M\r}\left.  - \frac{3}{4}\left((\g + 1)\r_+ - C\d \right)\left(\int_{y_0}^1 |w|^2\,dy - \frac{1}{1-y_0}\left(\int_{y_0}^1 w\,dy \right)^2 \right) \right] \\
    &\quad \leq -\frac{\g+1}{8}\r_+ \d \int_{y_0}^1 w^2 \,dy + \d \left[-\frac{1}{2}M\r_+^2 + \frac{3}{4}\frac{(\g+1)\r_+}{1-y_0}+C\d \right]\left(\int_{y_0}^1 w\,dy \right)^2\\
    &\quad \leq -\frac{\g+1}{8}\r_+\d \int_{y_0}^1 w^2 =  -\frac{\g+1}{8}\r_+ \intRp |u-\util|^2|\util'|\,dx = -\frac{\g+1}{8}\r_+G^S,
\end{align*}
where $G^S$ is defined in \eqref{terms}. 
Thus, we obtain
\begin{align*}
   & \dot{X}Y + B - \Gnew - G_1 - D_{u1} \\
    &\leq -\frac{\d}{4M}|\dot{X}|^2 + \frac{C}{\d}\big(|Y_1|^2 + |Y_2|^2 \big)\ + B - \frac{1}{2}\Gnew - G_1 - \frac{3}{4}D_{u1}\\
    &\quad -\frac{\d}{4M}|\dot{X}|^2 - \frac{1}{2}\Gnew - \frac{1}{4}D_{u1} \\
    &\leq -\frac{\d}{4M}|\dot{X}|^2 - \frac{1}{2}\Gnew -\frac{\g+1}{8}\r_+G^S - \frac{1}{4}D_{u1}.
\end{align*}
\end{proof}

\subsection{Estimates on the boundary terms}\label{sec:bd}
In this subsection, we provide an estimate for the boundary term $\mathcal{P}$.
\begin{lem}\label{lem:bd}
Under the hypothesis of Proposition~\ref{p-est}, there exists $C>0$ (independent of $\d$) such that
    \[
        \int_0^t \mathcal{P}\,ds \leq Ce^{-C\d \b} + C\e^2 \int_0^t \|(u-\util)_{xx}\|_{L^2(\Rp)}^2\,ds
    \]
    for all $t \in [0,T]$.
\end{lem}
\begin{proof}
First, recall from \eqref{small-X} that
\[
    -\s t - X(t) - \b \leq -\frac{\s}{2}t - \b < 0, \quad ^\forall t \leq T.
\]
$\bullet$ Estimate of  $\displaystyle{\int_0^t a(s,0)q(U;\Util)(s,0)\,ds}$: Recall from \eqref{id-ftn} that
\begin{equation}\label{bd main}
    \begin{aligned}
        &\int_0^t a(s,0)q(U;\Util)(s,0)\,ds \\
        &\, = \int_0^t a(s,0)\left[\frac{1}{2}\r u (u-\util)^2 + \frac{u}{\g - 1}p(\r|\rhotil) \right](s,0)\,ds + \int_0^t a(s,0)(p-\ptil)(s,0)(u-\util)(s,0)\,ds.
    \end{aligned}
\end{equation} 
Since $u(t,0) = u_- <0$, the first term on the right-hand side in \eqref{bd main} forms a good term. Thus, from \eqref{est:linf} and \eqref{small-X}, we have
\begin{align*}
    \int_0^t a(s,0)q(U;\Util)(s,0)\,ds &\leq \int_0^t a(s,0)(p-\ptil)(s,0)(u-\util)(s,0)\,ds\\
    &\leq C \| p -\ptil\|_{L^\infty(\Rp)} \int_0^t |u_- - \util|(s,0)\,ds\\
    &\leq C\e \int_0^t \d e^{-C\d |-\s s - X(s) - \beta|}\,ds \leq C\e e^{-C\d \beta}.
\end{align*}
$\bullet$ Estimate of $\displaystyle{\int_0^t a(s,0)(u-\util)(s,0)(u-\util)_x(s,0)\,ds}$: Using Interpolation inequality and \eqref{small-X}, we have
\begin{align*}
    &\left| \int_0^t a(s,0)(u-\util)(s,0)(u-\util)_x(s,0) \,ds \right|\\
    &\quad \leq \int_0^t |(u-\util)(s,0)|\norm{(u-\util)_x}_{L^2(\Rp)}^{1/2}\norm{(u-\util)_{xx}}_{L^2(\Rp)}^{1/2}\,ds\\
    &\quad\leq C\int_0^t  |(u-\util)(s,0)|^{4/3}\,ds + C\int_0^t \norm{(u-\util)_x}_{L^2(\Rp)}^{2}\norm{(u-\util)_{xx}}_{L^2(\Rp)}^{2}\,ds\\
    &\quad\leq C\int_0^t \d^{4/3}e^{-C\d|-\s s - X(s) - \beta|}\,ds + C\e^2 \int_0^t \norm{(u-\util)_{xx}}_{L^2(\Rp)}^{2}\,ds\\
    &\quad \leq C\d^{1/3}e^{-C\d \b} + C\e^2 \int_0^t \norm{(u-\util)_{xx}}_{L^2(\Rp)}^{2}\,ds.
\end{align*} 
Thus, we obtain
\[
    \int_0^t \mathcal{P}\,ds \leq Ce^{-C\d \b} + C\e^2 \int_0^t \|(u-\util)_{xx}\|_{L^2(\Rp)}^2\,ds.
\]
\end{proof}

\subsection{Proof of Lemma~\ref{zero}}
From Lemma~\ref{lem:XBG} and Lemma~\ref{lem:lead}, we have 
\begin{align*}
    \frac{d}{dt} \intRp a \eta(U|\Util)\,dx &\leq \dot{X}Y + B - G + \mathcal{P}\\
    &\leq -\frac{\d}{4M}|\dot{X}|^2 - \frac{1}{2}\Gnew -\frac{\g+1}{8}\r_+G^S - \frac{1}{4}D_{u1} + \mathcal{P}.
\end{align*}
Integrating the above inequality over $[0,t]$ for any $t \in [0,T]$, and using Lemma~\ref{lem:bd}, we obtain
\begin{align*}
    &\intRp \eta(U(t,x)|\Util(t,x))\,dx + \int_0^t \big(\d |\dot{X}|^2 + \Gnew + G^S + D_{u1} \big)\,ds\\
    &\quad \leq C \intRp  \eta(U(0,x)|\Util(0,x))\,dx  + Ce^{-C\d \b} +C\e^2 \int_0^t \|(u-\util)_{xx}\|_{L^2(\Rp)}^2\,ds.
\end{align*}
Since the relative entropy is equivalent to the $L^2$ distance under the small perturbation assumption, that is,
\[
     \|u - \util \|_{L^2(\Rp)}^2 + \|\rho - \rhotil \|_{L^2(\Rp)}^2\sim \intRp \eta \big(U(t,x)|\Util(t,x) \big)\,dx,
\]
we conclude that 
\begin{align*}
    &\norm{(\r - \rhotil, u-\util)(t,\cdot)}_{L^2(\Rp)}^2 + \int_0^t \big(\d |\dot{X}|^2 + \Gnew + G^S + D_{u1} \big)\,ds\\
    &\quad \leq C \norm{(\r - \rhotil, u-\util)(0,\cdot)}_{L^2(\Rp)}^2 + Ce^{-C\d \b} +C\e^2 \int_0^t \|(u-\util)_{xx}\|_{L^2(\Rp)}^2\,ds.
\end{align*}

\section{Higher order estimates}
\setcounter{equation}{0}
We prove the higher-order estimates for $U - \UtilX$. For simplicity, we define $\phi$ and $\psi$ by
\[
\phi(t,x) := \rho(t,x) - \rhotilX(t,x), \quad \psi(t,x)  := u(t,x) - \utilX(t,x).
\]
Then it is easy to check that $(\phi, \psi)$ satisfies
\begin{equation}
\begin{aligned}\label{NS-pert}
&\phi_t + u\phi_x + \r \psi_x = F + \dot{X}\rhotil_x,\\
&\r(\psi_t + u\psi_x) + p'(\r)\phi_x - \psi_{xx} = G + \dot{X}\big(\r \util_x\big),
\end{aligned}
\end{equation}
where
\begin{equation}\label{def:FG}
\begin{aligned}
    &F := -\psi \rhotil_x - \phi \util_x,\\
    &G := \rhotil_x\left[\phi \frac{p'(\rhotil)}{\rhotil} - (p'(\r) - p'(\rhotil))\right] - \r \psi \util_x - \frac{\phi}{\rhotil}\util_{xx}.
\end{aligned}
\end{equation}
First, we prove $H^1$ estimates for the perturbation related to the density: $\rho - \rhotilX$.
\begin{lem}\label{lem:phix}
Under the hypothesis of Proposition~\ref{p-est}, we have
\begin{align}\label{est:rhox}
\begin{aligned}
     &\|(\r - \rhotilX)\|_{H^1(\Rp)}^2 +  \|(u - \utilX)\|_{L^2(\Rp)}^2 \\
     &\quad +\int_0^t (\d |\dot{X}|^2 + \Gnew + G^S + G^{\text{bd}} + D_\r + D_{u1})\,ds\\
    &\leq C\left(\|(\r - \rhotil^{\b})(0,\cdot )\|_{H^1(\Rp)}^2 +  \|(u - \util^{\b})(0,\cdot)\|_{L^2(\Rp)}^2\right) +Ce^{-C\d \b} + C\e \int_0^t D_{u2}\,ds,
\end{aligned}
\end{align}
where $\Gnew, G^S, G^{\text{bd}}, D_\r, D_{u1}$ and $D_{u2}$ are defined in \eqref{terms}.
\end{lem}
\begin{proof}
Differentiating \eqref{NS-pert}$_1$, we have
\begin{equation}\label{eq:phitx}
    \begin{aligned}
        \phi_{tx} + u \phi_{xx} + \r \psi_{xx} &= F_x - \util_x \phi_x - \rhotil_x\psi_x - 2\phi_x \psi_x + \dot{X}\rhotil_{xx}\\
        &=: F_1 + \dot{X}\rhotil_{xx}.
    \end{aligned}
\end{equation}
Observe that $F_1$ satisfies the following estimates:
\begin{equation}\label{est:F1}
    \begin{aligned}
            |F_1| &= |(-\psi \rhotil_x - \phi \util_x)_x -\util_x \phi_x - \rhotil_x \psi_x - 2\phi_x \psi_x|\\
            &\leq C\big(\d|\util_x||(\psi, \phi)| + |(\phi_x, \psi_x)||\util_x| + |\phi_x \psi_x|\big).
    \end{aligned}
\end{equation}
Next, we divide \eqref{eq:phitx} by $\r$ to get
\begin{equation}\label{eq:phix}
\begin{aligned}
    \left(\frac{\phi_x}{\r}\right)_t + u \left(\frac{\phi_x}{\r}\right)_x + \psi_{xx} &= \frac{F_1}{\r} - \frac{\phi_x}{\r^2}(\r_t + \r_x u)+ \dot{X}\left(\frac{\rhotil_{xx}}{\r}\right)\\
    &=\frac{F_1}{\r} + \frac{\phi_x}{\r}u_x  + \dot{X}\left(\frac{\rhotil_{xx}}{\r}\right) =: F_2  + \dot{X}\left(\frac{\rhotil_{xx}}{\r}\right).
\end{aligned}
\end{equation}
Then, multiplying \eqref{eq:phix} by $\frac{\phi_x}{\r}$, and integrating it with respect to $x$, we obtain
\begin{equation}\label{est:phix}
    \begin{aligned}
        &\frac{1}{2}\frac{d}{dt}\norm{\frac{\phi_x}{\r}}_{L^2(\Rp)}^2 - \frac{1}{2}\intRp u_x \left(\frac{\phi_x}{\r}\right)^2\,dx - \left. \frac{1}{2}u\left(\frac{\phi_x}{\r}\right)^2\right|_{x=0} + \intRp \frac{\psi_{xx}\phi_x}{\r}\,dx\\
        &\quad = \intRp \frac{F_2 \phi_x}{\r}\,dx + \dot{X} \intRp \frac{\phi_x}{\r} \left(\frac{\rhotil_{xx}}{\r}\right)\,dx.
    \end{aligned}
\end{equation}
Now, we eliminate the term of $\psi_{xx}$ in \eqref{est:phix}. To this end, we multiply \eqref{NS-pert}$_2$ by $\frac{\phi_x}{\r}$, and integrating it with respect to $x$ to get
\[
    \intRp (\psi_t + u\psi_x)\phi_x\,dx + \intRp (p'(\r)\phi_x - \psi_{xx})\frac{\phi_{x}}{\r}\,dx = \intRp G \frac{\phi_x}{\r}\,dx + \dot{X} \intRp \phi_x \util_x\,dx.
\]
Using the identity:
\begin{align*}
    \intRp \psi_t \phi_x\,dx = \frac{d}{dt}\intRp \psi \phi_x\,dx - \intRp \phi_{xt} \psi \,dx =  \frac{d}{dt}\int \psi \phi_x\,dx + \left. (\psi 
    \phi_t)\right|_{x=0} + \intRp \psi_x \phi_t\,dx,
\end{align*}
we have
\begin{equation}\label{est:aux}
    \begin{aligned}
        &\frac{d}{dt}\intRp \psi \phi_x\,dx + \left. (\psi 
    \phi_t)\right|_{x=0} + \intRp \psi_x \phi_t\,dx + \intRp u\psi_x \phi_x\,dx \\
    &\quad \quad + \intRp \frac{p'(\r)}{\r}\phi_x^2\,dx - \intRp \psi_{xx}\frac{\phi_x}{\r}\,dx =  \intRp G \frac{\phi_x}{\r}\,dx + \dot{X} \intRp \phi_x \util_x\,dx.
    \end{aligned}
\end{equation}
Adding \eqref{est:aux} to \eqref{est:phix}, we obtain
\begin{equation*}
    \begin{aligned}
        &\frac{d}{dt}\left(\frac{1}{2}\norm{\frac{\phi_x}{\r}}_{L^2(\Rp)}^2  + \intRp \psi \phi_x \,dx\right)  \underbrace{\left.-\frac{1}{2}u_-\left(\frac{\phi_x}{\r}\right)^2\right|_{x=0} \,}_{=: G^\text{bd}} + \underbrace{\intRp\frac{p'(\r)}{\r}|\phi_x|^2\,dx}_{=: D_\r}\\
        &\quad = \frac{1}{2} \intRp u_x \left(\frac{\phi_x}{\r}\right)^2\,dx - \intRp \psi_x (\phi_t + u\phi_x) \,dx - \left. (\psi \phi_t) \right|_{x=0}\\
        &\qquad + \intRp \frac{\phi_x}{\r}(G + F_2)\,dx+  \dot{X} \intRp\left(\frac{\rhotil_{xx}}{\r^2}+\util_x \right) \phi_x\,dx\\
        &\quad =: I_1 + I_2 + I_3 + I_4 + I_5.
    \end{aligned}
    \end{equation*}
    Notice that $G^\text{bd}>0$ and $D_\r >0$ by  $u_-<0$ and $p'(\rho)>0$.\\   
$\bullet$ (Estimates of $I_1$) We use the Sobolev embedding and a priori assumption \eqref{p-assum} to have
\begin{equation}\label{I1 estimate}
\begin{aligned}
    |I_1| &\leq \frac{1}{2} \intRp |\psi_x| \left(\frac{\phi_x}{\r} \right)^2\,dx + \frac{1}{2}\intRp |\util_x| \left(\frac{\phi_x}{\r} \right)^2\,dx\\
    &\leq  \|\psi_x\|_{L^\infty(\Rp)}\|\phi_x\|_{L^2(\Rp)}^2+ C\delta^2 D_\r
    %&\leq \int_0^t \|\psi_x\|_{L^\infty(\Rp)} \intRp \left(\frac{\phi_x}{\r} \right)^2\,dx\,ds + C\delta^2 D_\r\\
    %&\leq C\int_0^t \|\psi_x\|_{H^1(\Rp)}\|\phi_x\|_{L^2(\Rp)}^2 + C\d^2 D_\r\\
    \leq C\e\|\psi_x\|_{H^1(\Rp)}\|\phi_x\|_{L^2(\Rp)} + C\d^2 D_\r\\
    &\leq C\e(D_{u1} + D_{u2}) + C(\d^2 + \e) D_\r.
\end{aligned}
\end{equation}
$\bullet$ (Estimates of $I_2$)
From \eqref{NS-pert}, $I_2$ term is equivalent to
\[
I_2 = \intRp \psi_x\big(-\r \psi_x + F + \dot{X}\rhotil_x\big)\,dx.
\]
By the definition of $F$ in \eqref{def:FG}, we have
\begin{equation*}
    |F| \leq C|\util_x|\big| (\phi, \psi)\big|.
\end{equation*}
Using \eqref{small-X} and Young's inequality, we obtain 
\begin{align*}
    |I_2| &\leq C   D_{u1}  +    \intRp |\psi_x||\util_x|\big| (\phi, \psi)\big|\,dx +   |\dot{X}| \intRp |\psi_x||\rhotil_x|\,dx \\
    &\leq  C   D_{u1}  +  \intRp |\util_x| |(\psi, \phi)|^2\,dx  + C\d   |\dot{X}|^2 .
\end{align*}
Moreover, we can control the quadratic term $|\util_x| |(\psi, \phi)|^2$ as follows:
\begin{equation}\label{phi quad}
\begin{aligned}
      \intRp |\util_x| |(\psi, \phi)|^2 dx &\leq   \intRp\left(\left(\r - \rhotil - \frac{\rhotil}{(\s - \util)}(u-\util) \right)^2+  C|(u-\util)|^2\right) |\util_x| dx \\
    &\leq C\sqrt{\d}  \Gnew   + C   G^S .
\end{aligned}
\end{equation}
Therefore, we obtain
\begin{align*}
    |I_2| \leq C   D_{u1}  + C\sqrt{\d}  \Gnew  + C  G^S   + C\d   |\dot{X}|^2 .
\end{align*}
$\bullet$ (Estimates of $I_3$)
Recall from \eqref{NS-pert} and \eqref{def:FG} that $\phi$ satisfies
\[
    \phi_t = -u \phi_x - \r \psi_x -\psi \rhotil_x - \phi \util_x + \dot{X}\rhotil_x.
\]
This implies
\begin{equation}\label{est:phit}
    |\phi_t| \leq C\big[|\psi_x| + |\phi_x| + |\psi \rhotil_x| + |\phi \util_x| + |\dot{X}||\rhotil_x| \big].
\end{equation}
For $I_3$, we use \eqref{est:phit} and Young's inequality to have
\begin{equation*}
    \begin{aligned}
        \int_0^t |I_3|\,ds &= \left|\int_0^t \left. \psi \phi_t \right|_{x=0}\,ds \right| \leq C\left[\int_0^t  |\psi \psi_x| \Big|_{x=0}\,ds + \int_0^t |\psi \phi_x| \Big|_{x=0}\,ds + \int_0^t \psi^2 |\rhotil_x| \Big|_{x=0}\,ds \right.\\
        &\phantom{|I_4| = \left|\int_0^t \left. \psi \phi_t \right|_{x=0}\,ds \right| \leq C[\int_0^t  |\psi|} \left. +\int_0^t |\psi \phi||\util_x|\Big|_{x=0}\,ds + \int_0^t |\dot{X}||\psi||\rhotil_x|\Big|_{x=0}\,ds  \right]\\
        &\leq C\int_0^t  |\psi \psi_x| \Big|_{x=0}\,ds + \frac{1}{40}\int_0^t G^{\text{bd}}\,ds + C\int_0^t |\psi|^2\Big|_{x=0}\,ds + C\e^2 \int_0^t |\util_x|\Big|_{x=0}\,ds.
    \end{aligned}
\end{equation*}
Following the similar calculations as in Lemma~\ref{lem:bd} with Lemma~\ref{lem:VS}, we find that
\[
    \int_0^t|I_3|\,ds \leq \frac{1}{40}\int_0^t G^{\text{bd}}\,ds + C\d^{1/3}e^{-C\d \b} + C\e^2 \int_0^t D_{u2}\,ds + C\e e^{-C\d \b} + C\e^2 \d e^{-C\d \b}.
\]
$\bullet$ (Estimates of $I_4$)
By the definition of $G$ in \eqref{def:FG}, observe that
\begin{equation}\label{est:G}
|G| \leq C|\util_x||(\phi, \psi)|.
\end{equation}
\begin{comment}
Moreover, we can control the quadratic term $|\util_x| |(\psi, \phi)|^2$ as follows:
\begin{equation}\label{phi quad}
\begin{aligned}
    &\int_0^t \intRp |\util_x| |(\psi, \phi)|^2\,dx\,ds \\
    &\leq \int_0^t \intRp |\util_x|\left(\r - \rhotil - \frac{\rhotil}{(\s - \util)}(u-\util) \right)^2+  C|\util_x||(u-\util)|^2\,dx\,ds\\
    &\leq C\frac{\d}{\l}\int_0^t \Gnew \,ds + C \int_0^t G^S\,ds.
\end{aligned}
\end{equation}
\end{comment}
Using \eqref{phi quad} and Young's inequality, we get
\begin{equation*}
\begin{aligned}
 \intRp \left|\frac{\phi_x}{\r}G\right|dx \leq C \intRp |\util_x||(\phi, \psi)||\phi_x|dx \leq 
\frac{1}{40}D_\r + C\sqrt{\d}\Gnew + CG^S.
\end{aligned}
\end{equation*}
In order to estimate $\intRp \frac{\phi_x}{\r}F_2 \,dx$, we use \eqref{I1 estimate}, \eqref{est:F1} and \eqref{phi quad} to have
\begin{equation}\label{est:F2}
\begin{aligned}
    \left|  \intRp \frac{\phi_x}{\r}F_2 \,dx \right|  &\leq \left|  \intRp \frac{\phi_x}{\r}\frac{u_x\phi_x}{\r}\,dx \right| + \left|  \intRp \frac{\phi_x}{\r}\frac{F_1}{\r} \,dx \right|\\
    &\leq C\e  (D_{u1}+ D_{u2})  + C(\d^2+\e)  D_\r  + C\d   \intRp |\phi_x||(\phi, \psi)||\util_x|\,dx \\
    &\qquad + C  \intRp |\phi_x||(\phi_x, \psi_x)||\util_x|\,dx  +C  \intRp \phi_x^2|\psi_x|\,dx \\
    &\leq \frac{1}{40}D_\r + C\d\big(\sqrt{\d}\Gnew + G^S \big) + C\d^2(D_\r + D_{u1}) + C\e (D_{u1}+ D_{u2}) .
    %&\, \leq \frac{1}{40}  D_\r ds + C\d\left[  (\sqrt{\d}\Gnew + G^S) \right] + C\d^2  (D_\r + D_{u1}) + C\e  (D_{u1}+ D_{u2})ds.
\end{aligned}
\end{equation}
Thus, we obtain
\[
|I_4| \leq   \left(\frac{1}{20} + C\d^2 \right)D_\r + C\sqrt{\d}\Gnew + CG^S + C(\d^2 + \e)D_{u1} + C\e D_{u2} .
\]
$\bullet$ (Estimates of $I_5$)
The $I_5$ term can be estimated as follows:
\begin{align*}
    |I_5| &\leq C  |\dot{X}|\intRp |\phi_x||\util_x| \,dx   \leq   |\dot{X}|^2\intRp |\util_x|\,dx  + C  \intRp |\phi_x|^2|\util_x|\,dx \\
    &\leq \d   |\dot{X}|^2  + C\d^2   D_\r .
\end{align*}
Combining all those estimates and integrating \eqref{est:psix} with respect to $t$, we obtain
\begin{equation}\label{phi x estimate}
    \begin{aligned}
        &\|\phi_x\|_{L^2(\Rp)}^2 + \int_0^t D_\r \,ds+ \int_0^t G^{\text{bd}}\,ds\\
        &\, \leq C_1 \left[\int_0^t \d |\dot{X}|^2\,ds + \int_0^t G^S\,ds + \int_0^t D_{u1}\,ds + \intRp \psi^2\,dx \right]\\
        &\quad  +C\left[\|\phi_x(0,x)\|_{L^2(\Rp)}^2 + \|(\phi, \psi)(0,x)\|_{L^2(\Rp)}^2 + \sqrt{\d}\int_0^t \Gnew\,ds + \e \int_0^t D_{u2}\,ds + e^{-C\d \b}  \right],
    \end{aligned}
\end{equation}
for some positive constant $C_1$. Multiplying \eqref{phi x estimate} by $\frac{1}{2 \max\{1, C_1\}}$ and adding it to \eqref{est:zero}, we obtain the desired result in Lemma~\ref{lem:phix}.
%\[
%  \begin{aligned}
%    &\|(\r - \rhotilX)\|_{H^1(\Rp)}^2 +  \|(u - \utilX)\|_{L^2(\Rp)}^2 \\
%    &\quad +\int_0^t (\d |\dot{X}|^2 + \Gnew + G^S + G^{\text{bd}} + D_\r + D_{u1})\,ds\\
%    &\leq C\left(\|(\r - \rhotil^{\b})(0,\cdot)\|_{H^1(\Rp)}^2 +  \|(u - \util^{\b})(0,\cdot)\|_{L^2(\Rp)}^2\right) +Ce^{-C\d \b} + C\e \int_0^t D_{u2}\,ds.
%  \end{aligned}  
%\]
\end{proof}
Now, we complete the proof of Proposition~\ref{p-est} by using the following lemma.
\begin{lem}\label{lem:psix}
    Under the hypothesis in \eqref{p-est}, we have
    \[
        \begin{aligned}
            &\|(\r - \rhotilX)\|_{H^1(\Rp)}^2 +  \|(u - \utilX)\|_{H^1(\Rp)}^2 \\
            &\quad +\int_0^t (\d |\dot{X}|^2 + \Gnew + G^S + G^{\text{bd}} + D_\r + D_{u1} + D_{u2})\,ds\\
            &\leq C\left(\|(\r - \rhotil^{\b})(0,\cdot)\|_{H^1(\Rp)}^2 +  \|(u - \util^{\b})(0,\cdot)\|_{H^1(\Rp)}^2\right) +Ce^{-C\d \b},
        \end{aligned}
    \]
    where $\Gnew, G^S, G^{\text{bd}}, D_\r, D_{u1}$ and $D_{u2}$ are defined in \eqref{terms}.
    \end{lem}
    \begin{proof}
    Multiplying \eqref{NS-pert}$_2$ by $-\frac{\psi_{xx}}{\r}$, and integrating with respect to $x$, we have
    \begin{equation}\label{est:psix}
        \begin{aligned}
        \frac{1}{2}\|\psi_x\|_{L^2(\Rp)}^2  + \intRp \frac{|\psi_{xx}|^2}{\r}\,dx &= - \psi_t(t,0) \psi_x(t,0) + \intRp \left(\frac{p'(\r)}{\r}\phi_x + u \psi_x \right)\psi_{xx}\,dx\\
        &\quad -\intRp \frac{G \psi_{xx}}{\r}\,dx - \dot{X} \intRp \psi_{xx}\util_x\,dx\\
        &=: J_1 +J_2 + J_3 +J_4.
        \end{aligned}
    \end{equation}
    Here, we use the identity:
    \[
        -\intRp \psi_t \psi_{xx} = \frac{1}{2}\frac{d}{dt}\|\psi_x\|_{L^2(\Rp)}^2 + \psi_t(t,0) \psi_x(t,0).
    \]
    To estimate the $J_1$ term, observe that
    \begin{equation}\label{u t boundary}
            |(u-\util)_t (t,0)| = |\util_t(t,0)| \leq C |\util'(-\s t - X(t) - \b)||\s + \dot{X}(t)| \leq C \d^2 e^{-C\d(\s t + \b)}.
    \end{equation}
    Using Young's inequality and \eqref{u t boundary} , we obtain
    \begin{equation}\label{J_1 estimate}
        \begin{aligned}
            \int_0^t|J_1|\,ds &= \left|\int_0^t (\util_t \psi_x)(s,0)\,ds \right| \leq C\int_0^t \|\psi_x\|_{L^\infty(\Rp)}^2 |\util_t(s,0)|\,ds + \int_0^t |\util_t(s,0)|\,ds\\
            %%\int_0^t \|\psi_x\|_\infty |\util_t(s,0)|\,ds\\
            &\leq C\d^2 \int_0^t \|\psi_x\|_{L^2(\Rp)}\|\psi_{xx}\|_{L^2(\Rp)}+\int_0^t |\util_t(s,0)|\,ds\\
            &\leq C\d^2\int_0^t (D_{u1}+D_{u2})\,ds + C\d e^{-C \d \b}.
        \end{aligned}
    \end{equation}
    For $J_2$, we use \eqref{phi quad} and Young's inequality to obtain
    \[
        |J_2| \leq \frac{1}{40} D_{u2} + C (D_\r + D_{u1}).
    \]
    From \eqref{est:G}, the term $J_3$ can be estimated as
    \[
       |J_3| \leq \intRp |\util_x||(\phi,\psi)||\psi_{xx}|\,dx \leq C\d^2 D_{u2} + C\sqrt{\d} \Gnew + C G^S . 
    \]
    Finally, for the term $J_4$, we have
    \[
        |J_4| \leq C|\dot{X}|\intRp |\util_x||\psi_{xx}|\,dx\leq C\d  |\dot{X}|^2 + C\d^2 \int_0^t D_{u2}.
    \]
    Using these estimates, integrating \eqref{est:psix} with respect to $t$, and applying the fact
    \[
        \intRp \frac{|\psi_{xx}|^2}{\r}\,dx \sim D_{u2},
    \]
    we obtain 
    \begin{equation}\label{psi x estimate}
        \begin{aligned}
            \|\psi_x\|_{L^2(\Rp)}^2 + \int_0^t D_{u2}\,ds &\leq C_2 \left(\int_0^t \d |\dot{X}|^2\,ds + \int_0^t G^S \,ds + \int_0^t D_\r\,ds + \int_0^t D_{u1}\,ds\right) \\
            &\quad + C\left(\|\psi_x(t,0)\|_{L^2(\Rp)}^2 + \sqrt{\d}\int_0^t \Gnew \,ds + e^{-C\d \b }\right),
        \end{aligned}
    \end{equation}
    for some positive constant $C_2$. Hence, multiplying \eqref{psi x estimate} by $\frac{1}{2 \max\{1, C_2\}}$ and adding it to \eqref{est:rhox}, we obtain the desired result in Lemma~\ref{lem:psix}.
    \end{proof}

\appendix

\section{Proof of Theorem \ref{thm:main}}\label{Appenidx B}
\setcounter{equation}{0}
\subsection{Proof of global existence}
Here, we prove the global existence of solutions to the outflow problem, based on the local existence and the a priori estimates.

First, we choose smooth monotone functions $\rubar$ and $\uubar$ on $\Rp$ satisfying
\begin{align*}
    &\norm{\rubar - \r_+}_{L^2(\b, \infty)} + \norm{\rubar - \r_-}_{L^2(0, \b)}\\
    &\quad +\norm{\uubar - u_+}_{L^2(\b, \infty)} + \norm{\uubar - u_-}_{L^2(0, \b)} + \norm{\rubar_x}_{L^2(\Rp)} + \norm{\uubar_x}_{L^2(\Rp)} \leq \underline{C}\delta
\end{align*}
for some constant $\underline{C}$. Since $(\rhotilX, \utilX)(0,x) = (\rhotil, \util)(x-\b)$, we have 
\begin{align*}
    &\norm{\rubar(\cdot) - \rhotilX(0,\cdot)}_{H^1(\Rp)} + \norm{\uubar(\cdot) - \utilX(0,\cdot)}_{H^1(\Rp)}\\
    &\quad \leq \norm{\rubar - \r_+}_{L^2(\b, \infty)} + \norm{\rubar - \r_-}_{L^2(0, \b)}+\norm{\uubar - u_+}_{L^2(\b, \infty)} + \norm{\uubar - u_-}_{L^2(0, \b)}\\
    &\quad \quad +\norm{\rhotilX(0,\cdot) - \r_+}_{L^2(\b, \infty)} + \norm{\rhotilX(0,\cdot) - \r_-}_{L^2(0, \b)}\\
    &\quad \quad +\norm{\utilX(0,\cdot) - u_+}_{L^2(\b, \infty)} + \norm{\utilX(0,\cdot) - u_-}_{L^2(0, \b)}\\
    &\quad \quad + \norm{\rubar_x}_{L^2(\Rp)} + \norm{\rd_x \rhotilX(0,\cdot)}_{L^2(\Rp)} + \norm{\uubar_x}_{L^2(\Rp)} + \norm{\rd_x \utilX(0,\cdot)}_{L^2(\Rp)}\\
    &\leq \overline{C}\sqrt{\d}
\end{align*}
for some constant $\overline{C}$. Now, define positive constants $\e_0$ and $\e_*$ as
\begin{equation}\label{smalle}
    \e_0 = \e_* - \underline{C}\d, \quad \e_* := \frac{\e}{2(C_0 + 1)} - \overline{C} \sqrt{\d} - e^{-C\d \b},
\end{equation}
where $\e$ and $C_0$ are constants defined in Theorem \ref{p-est}. Here, thanks to the smallness of $\d$ and $e^{-C\d \b}$, we can choose $\e_0$ and $\e_*$ as positive constants and $\e_* < \frac{\e}{2}$. Moreover, note that $\e_0$ can be chosen independently on $\d$, for example, as $\e_0 = \frac{\e}{4(C_0+1)}$.

Now, consider any initial data $(\r_0, u_0)$ that satisfies \eqref{initial perturbation}, that is, 
\[
    \norm{(\r_0, u_0) - (\r_+, u_+)}_{L^2(\b, \infty)} + \norm{(\r_0, u_0) - (\r_-, u_-)}_{L^2(0, \b)} + \norm{(\rd_x \r_0, \rd_x u_0)}_{L^2(\Rp)} < \e_0.
\]
This implies that 
\begin{align*}
    &\norm{\r_0 - \rubar}_{H^1(\Rp)} + \norm{u_0 - \uubar}_{H^1(\Rp)} \\
    &\quad \leq \norm{(\r_0, u_0) - (\r_+, u_+)}_{L^2(\b, \infty)} + \norm{(\rubar, \uubar) - (\r_+, u_+)}_{L^2(\b, \infty)}\\
    &\quad \quad  + \norm{(\r_0, u_0) - (\r_-, u_-)}_{L^2(0, \b)} + \norm{(\rubar, \uubar) - (\r_-, u_-)}_{L^2(0, \b)}\\
    &\quad \quad + \norm{(\rd_x \r_0, \rd_x u_0)}_{L^2(\Rp)} + \norm{(\rd_x \rubar, \rd_x \uubar)}_{L^2(\Rp)}\\
    &\quad \leq \e_0 + \underline{C}\d = \e_*.
\end{align*}
In particular, Sobolev inequality implies $\norm{\r_0 - \rubar}_{L^\infty(\Rp)} < C\e_*$. Therefore, for $\e_*$ sufficiently small, we obtain
\[
    \frac{\r_+}{2} < \r_0(x) < 2\r_+, \quad x \in \Rp.
\]
Hence, by Theorem \ref{prop:loc}, there exists $T_0>0$ such that the system \eqref{NS} admits a unique solution $(\r, u)$ on $[0,T_0]$ satisfying 
\begin{equation}\label{Local smallness 1}
    \norm{(\r - \rubar, u - \uubar)}_{L^\infty(0,T_0;H^1(\Rp))} \leq \frac{\e}{2},
\end{equation}
and
\[
\frac{\r_-}{3} < \r(t,x) < 3\r_+, \quad t>0, \quad x \in \Rp.
\]
On the other hand, observe that 
\begin{align*}
    &\norm{\rubar - \rhotilX}_{H^1(\Rp)} + \norm{\uubar - \utilX}_{H^1(\Rp)}\\
    &\quad \leq  \norm{(\rubar, \uubar) - (\r_+, u_+)}_{L^2(\b, \infty)} + \norm{(\rhotilX, \utilX) - (\r_+, u_+)}_{L^2(\b, \infty)} \\
    &\quad \quad  + \norm{(\rubar, \uubar) - (\r_-, u_-)}_{L^2(0, \b)} + \norm{(\rhotilX, \utilX) - (\r_-, u_-)}_{L^2(0, \b)}\\
    &\quad \quad \norm{(\rd_x \rubar, \rd_x \uubar)}_{L^2(\Rp)} + \norm{(\rd_x \rhotilX, \rd_x \utilX)}_{L^2(\Rp)}\\
    &\quad \leq C\sqrt{\d}\left(1+ \sqrt{|X(t)|}\right) \leq C\sqrt{\d}(1+\sqrt{t}).
\end{align*}
Here, the details of the above inequality can be found in \cite{KVW23}. Therefore, if we choose $\d>0$ and $T_1 \in (0,T_0)$ sufficiently small so that
\[
C\sqrt{\d}(1+\sqrt{T_1}) < \frac{\e}{2},
\]
we get
\[
    \norm{(\rubar - \rhotilX, \uubar - \utilX)}_{L^\infty(0,T_1;H^1(\Rp))} \leq \frac{\e}{2}.
\]
Combining this with \eqref{Local smallness 1}, we then have 
\[
    \norm{(\r - \rhotilX, u - \utilX)}_{L^\infty(0,T_1;H^1(\Rp))} \leq \e.
\]
Moreover, since the shift function $X(t)$ is Lipschitz, and 
\[
\r - \rubar, \, \, u - \uubar \in C([0,T_1];H^1(\Rp)),
\]
we obtain $\r - \rhotilX, \, \, u - \utilX \in C([0,T_1];H^1(\Rp))$.

Until now, we proved that \eqref{p-assum} holds in the time interval $[0,T_1]$. We now show that the solution can be extended globally by using the continuation argument.

Consider the maximal existence time:
\[
T_M := \sup \left\{t>0 : \sup_{t \in [0,T]} \left(\norm{(\r - \rhotilX, u - \utilX)}_{H^1(\Rp)} \right) \leq \e \right\}.
\]
We show that $T_M = \infty$. If $T_M < \infty$, then we have
\begin{equation}\label{conti argument}
    \sup_{t \in [0,T_M]} \left(\norm{(\r - \rhotilX, u - \utilX)}_{H^1(\Rp)} \right) = \e.
\end{equation}
On the other hand, since 
\begin{align*}
    &\norm{(\r_0, u_0)(\cdot) - (\rhotilX, \utilX)(0,\cdot)}_{H^1(\Rp)}\\
    &\quad \leq \norm{(\r_0, u_0) - (\rubar, \uubar)}_{H^1(\Rp)} + \norm{(\rubar, \uubar)(\cdot) - (\rhotilX, \utilX)(0,\cdot)}_{H^1(\Rp)}\\
    &\quad < \e_* + \overline{C}\sqrt{\d},
\end{align*}
a priori estimate \eqref{est:pri} yields
\begin{align*}
    &\sup_{t \in [0,T_M]} \left(\norm{(\r - \rhotilX, u - \utilX)}_{H^1(\Rp)} \right) \\
    &\qquad \leq C_0 \left( \| \r_0 - \rhotil^{\beta} \|_{H^1(\mathbb{R}_+)} + \| u_0 - \util^{\beta} \|_{H^1(\mathbb{R}_+)} \right) + C_0 e^{-C\delta \beta}\\
    &\qquad \leq C_0(\e_* + \overline{C} \sqrt{\d}) + C_0 e^{-C\d \b} < \frac{\e}{2},
\end{align*}
which contradicts to \eqref{conti argument}.

Hence, we obtain $T_M = +\infty$. Moreover, the a priori estimate holds for the whole time interval:
\begin{equation}\label{global in time estimate}
    \begin{aligned}
        &\sup_{t >0} \left[ \| \r - \rhotil^{X, \beta} \|_{H^1(\mathbb{R}_+)} + \| u - \util^{X, \beta} \|_{H^1(\mathbb{R}_+)}\right] \\
        &\qquad \qquad + \sqrt{\int_0^\infty \left(\delta |\dot{X}|^2 + G^{\text{new}} + G^S + D_\r + D_{u1} + D_{u2}\right)\,ds}\\
        &\qquad \leq C_0 \left( \| \r_0 - \rhotil^{\beta} \|_{H^1(\mathbb{R}_+)} + \| u_0 - \util^{\beta} \|_{H^1(\mathbb{R}_+)} \right) + C_0 e^{-C\delta \beta}.
        \end{aligned}
\end{equation}
Furthermore, for all $t>0$, we have 
\begin{equation}\label{est:X}
    |\dot{X}(t)| \leq C_0\left(\| \r - \rhotil^{X, \beta} \|_{L^\infty(\mathbb{R}_+)} + \| u - \util^{X, \beta} \|_{L^\infty(\mathbb{R}_+)}  \right) \leq C\e.
\end{equation}
\subsection{Time-asymptotic stability}
We now show the long-time behavior \eqref{asym-U} of the solution to the outflow problem using the global-in-time estimate \eqref{global in time estimate}.

For this purpose, we define $g(t)$ by
\[
    g(t) := \|(u-\utilX)_x\|_{L^2(\Rp)}^2 + \|(\r-\rhotilX)_x\|_{L^2(\Rp)}^2.
\]
We aim to show that  
\[
   \int_0^\infty |g(t)| + |g'(t)| \,dt < + \infty. 
\]
This implies
\[
    g(t) \to 0 \quad \text{as} \quad t \to \infty.
\]
Hence, the interpolation inequality yields 
\[
    \|(u - \utilX)\|_{L^\infty(\Rp)} \leq C \underbrace{\|(u - \utilX)\|_{L^2(\Rp)}^{1/2}}_{\text{Bounded}} \underbrace{\|(u - \utilX)_x\|_{L^2(\Rp)}^{1/2}}_{\to 0} \to 0 \quad \text{as} \quad t \to \infty. 
\]
Similarly, we have 
\[
    \|(\r - \rhotilX)\|_{L^\infty(\Rp)} \to 0 \quad \text{as} \quad t \to \infty.
\]
In what follows, we prove that $g \in W^{1,1}(0,\infty)$. First, note that 
\begin{align*}
    \int_0^\infty |g(t)| dt \le C\int_0^\infty (D_\r + D_{u1})dt < \infty.
\end{align*}
Therefore, we have $g \in L^1(0,\infty)$. Next, observe that 
\begin{align*}
    \frac{1}{2}\int_0^\infty |g'(t)|\,dt &\leq \int_0^\infty \left|  \intRp (u-\util)_x (u-\util)_{xt} \,dx \right|\,dt  + \left|  \intRp (\r-\rhotil)_x (\r-\rhotil)_{xt} \,dx  \right|\,dt \\
    &=: I_1 + I_2.
\end{align*}
For $I_1$, following the calculations in \eqref{J_1 estimate}, we have
\begin{align*}
    |I_1| &\leq \int_0^\infty  \left|(u-\util)_x (u-\util)_{t} \right|\Big|_{x=0}\,dt + \int_0^\infty \left| \intRp (u-\util)_{xx}(u-\util)_t\,dx\right|\,dt\\
    % &\leq C\int_0^\infty \norm{(u-\util)_x}_{L^\infty(\Rp)}\,dt + C \int_0^\infty D_{u2}\,dt + C\int_0^\infty \intRp |(u-\util)_t|^2\,dx\,dt\\
    % &\leq  C\left[\int_0^\infty \norm{(u-\util)_x}_{L^2(\Rp)}^{1/2}\norm{(u-\util)_{xx}}_{L^2(\Rp)}^{1/2}dt +  \int_0^\infty \left(D_{u2} + \intRp |(u-\util)_t|^2 dx\right) dt\right]\\
     &\leq C+ C\int_0^\infty (D_{u1} + D_{u2})\,dt + C\int_0^\infty \intRp |(u-\util)_t|^2\,dx\,dt.
\end{align*}
From \eqref{NS-pert}$_2$ and \eqref{est:G}, we obtain
\begin{align*}
    \int_0^\infty \intRp |(u-\util)_t|^2\,dx\,dt &\leq C\int_0^{\infty} \left[D_{\r} + D_{u1} + D_{u2} + G^S + \Gnew\right]\,dt  + C\delta \int_0^\infty |\dot{X}|^2\,dt \\
    &< \infty.
\end{align*}
For $I_2$, recall from \eqref{eq:phitx} and \eqref{est:F1} that $\r -\rhotil = \phi$ satisfies
\[
\phi_{tx} = -u \phi_{xx} - \rho \psi_{xx} + F_1 + \dot{X} \rhotil_{xx},
\]
and $F_1$ satisfies 
\[
|F_1| \leq C\big(\d|\util_x||(\psi, \phi)| + |(\phi_x, \psi_x)||\util_x| + |\phi_x \psi_x|\big).
\]
Thus, we have
\begin{align*}
    |I_2| = \int_0^\infty \left|\intRp \phi_x \phi_{xt}\,dx \right|\,dt &\leq \int_0^\infty \left|\intRp u\phi_{xx} \phi_{x}\,dx \right|\,dt + \int_0^\infty \left|\intRp \r \psi_{xx} \phi_{x}\,dx \right|\,dt\\
    &\quad + \int_0^\infty \left|\intRp F_1 \phi_{x}\,dx \right|\,dt + \int_0^\infty |\dot{X}|\left|\intRp \phi_x  \rhotil_{xx} \right|\,dt\\
    &=: J_1 + J_2 + J_3 + J_4.
\end{align*}
Following the arguments in \eqref{I1 estimate}, the $J_1$ term can be estimated as follows:
\begin{align*}
    |J_1| &\leq \int_0^\infty \left|\frac{1}{2}u(\phi_{x})^2\Big|_{x=0} \right|\,dt + \int_0^\infty \left| \intRp \frac{1}{2}u_x (\phi_x)^2\,dx\right|\,dt\\
    &\leq C \int_0^\infty \left[ G^{\text{bd}} + D_{u1} + D_{u2} + D_\r\right]\,dt.
\end{align*}
For the term $J_2$, we get
\[
|J_2| \leq C\int_0^\infty (D_\r + D_{u2})\,dt.
\]
Next, by repeating the process in \eqref{est:F2}, we find that
\begin{align*}
    |J_3| \leq C \int_0^\infty \big( D_\r + D_{u1} + D_{u2} + G^{\text{new}} + G^S \big)\,dt.
\end{align*}
Finally, using Lemma~\ref{lem:VS}, we have
\begin{align*}
    |J_4| &\leq C\int_0^\infty |\dot{X}| \intRp |\phi_x||\util_x| \,dx\,dt \leq C\int_0^\infty  \intRp \left(|\phi_x|^2 + |\dot{X}|^2|\util_x|^2\right)\,dx\,dt\\
    &\leq  C \delta \int_0^\infty |\dot{X}|^2\,dt + C\int_0^\infty D_\r\,dt.
\end{align*}
Hence, we obtain
\[
\int_0^\infty |g'(t)|\,dt \leq C\int_0^\infty \big(G^{\text{new}} + G^S + \d |\dot{X}|^2 + D_{u1} + D_{u2} + D_\r + G^{\text{bd}})\,dt < \infty.
\]
\subsection{Proof of \eqref{asym-X}}
Finally, the asymptotic behavior \eqref{asym-X} of $\dot{X}(t)$ easily follows from \eqref{asym-U} and \eqref{est:X} as follows:
\[
|\dot{X}(t)| \leq C_0\left(\| \r - \rhotil^{X, \beta} \|_{L^\infty(\mathbb{R}_+)} + \| u - \util^{X, \beta} \|_{L^\infty(\mathbb{R}_+)}  \right) \to 0 \quad \text{as} \quad t \to 0.
\]
\qed
\vspace{0.5cm}
\bibliography{reference.bib}
\end{document}